\documentclass[letterpaper, 10pt]{IEEEtran}

\usepackage{cite}
\usepackage{amsmath,amssymb,amsfonts}
\usepackage{amsthm}
\usepackage{hyperref}
\usepackage{epstopdf,color}
\usepackage{textcomp}
\usepackage{graphicx,color}
\graphicspath{{img/}}
\usepackage{mathrsfs}
\usepackage[vlined,ruled]{algorithm2e}
\usepackage{subfigure}
\usepackage{url}
\usepackage{color}
\usepackage{dsfont}
\usepackage{bbm}
\usepackage{booktabs}
\usepackage{array}
\usepackage[table]{xcolor}
\usepackage{yfonts}
\usepackage{cleveref} 
\usepackage{tikz}
\usepackage{pgfplots}
\usepackage{multirow}
\usepackage{textcomp}
\usepackage{tabularx}
\usepackage{placeins}
\usepackage{float}
\usepackage{nccmath}
\usepackage{accents}
\usepackage{multicol}
\usepackage{etoolbox, refcount}
\usepackage{chngcntr}
\usepackage{apptools}
\usepackage{relsize}
\usepackage[font=footnotesize]{caption}

\newtheorem{theorem}{Theorem}[section]

\newtheorem{definition}[theorem]{Definition}

\newtheorem{proposition}[theorem]{Proposition}

\newtheorem{remark}[theorem]{Remark}

\newtheorem{example}[theorem]{Example}

\newcommand{\setdef}[2]{\{#1 : #2\}}

\newcommand{\Kc}{\mathcal{K}}
\newcommand{\Tc}{\mathcal{T}}

\newcommand{\Uc}{\mathcal{U}}

\newcommand{\real}{\mathbb{R}}

\newcommand{\integerspos}{\mathbb{Z}_{> 0}}

\newcommand{\pder}[2]{\frac{\partial #1}{\partial #2}}

\newcommand{\C}{\mathcal{C}}

\newcommand{\Cc}{\mathcal{C}}
\newcommand{\Dc}{\mathcal{D}}
\newcommand{\Sc}{\mathcal{S}}
\newcommand{\Pc}{\mathcal{P}}
\newcommand{\Vc}{\mathcal{V}}

\DeclareSymbolFont{bbold}{U}{bbold}{m}{n}
\DeclareSymbolFontAlphabet{\mathbbold}{bbold}

\newcommand{\norm}[1]{\left\lVert#1\right\rVert}

\newcommand{\cbf}{\operatorname{cbf}}

\newcommand{\Int}{\operatorname{Int}}
\newcommand{\clos}{\operatorname{clos}}

\newcommand\oprocendsymbol{\hbox{$\bullet$}}
\newcommand\oprocend{\relax\ifmmode\else\unskip\hfill\fi\oprocendsymbol}

\renewcommand{\theenumi}{(\roman{enumi})}

\newcommand*{\QEDA}{\hfill\ensuremath{\blacksquare}}%

\newcommand\xqed[1]{%
  \leavevmode\unskip\penalty9999 \hbox{}\nobreak\hfill
  \quad\hbox{#1}}
\newcommand\demo{\xqed{$\bullet$}}

\newcommand\problemfinal{\xqed{$\triangle$}}

\newcounter{countitems}
\newcounter{nextitemizecount}
\newcommand{\setupcountitems}{%
  \stepcounter{nextitemizecount}%
  \setcounter{countitems}{0}%
  \preto\item{\stepcounter{countitems}}%
}
\makeatletter
\newcommand{\computecountitems}{%
  \edef\@currentlabel{\number\c@countitems}%
  \label{countitems@\number\numexpr\value{nextitemizecount}-1\relax}%
}
\newcommand{\nextitemizecount}{%
  \getrefnumber{countitems@\number\c@nextitemizecount}%
}
\newcommand{\previtemizecount}{%
  \getrefnumber{countitems@\number\numexpr\value{nextitemizecount}-1\relax}%
}
\makeatother    
{\end{itemize}%
\unskip\computecountitems\ifnumcomp{\previtemizecount}{>}{4}{\end{multicols}}{}}

\newcommand{\longthmtitle}[1]{\mbox{}\emph{(#1):}}

\renewcommand{\baselinestretch}{1.03}
\newcommand{\comment}[1]{} 
\newcolumntype{P}[1]{>{\centering\arraybackslash}p{#1}}
\allowdisplaybreaks

\begin{document}
\title{\LARGE \bf Converse Theorems for Certificates of Safety and
  Stability}

\author{Pol Mestres \qquad Jorge Cort{\'e}s \thanks{P. Mestres and
    J. Cort\'es are with the Department of Mechanical and Aerospace
    Engineering, University of California, San Diego,
    \{pomestre,cortes\}@ucsd.edu}%
}

\maketitle

\begin{abstract}
  Motivated by the key role of control barrier functions (CBFs) in
  assessing safety and enabling the synthesis of safe controllers in
  nonlinear control systems, this paper presents a suite of converse
  results on CBFs.  Given any safe set, we first
  identify a set of general sufficient conditions which guarantee the
  existence of a CBF. Our technical analysis also enables us to define
  an extended notion of CBF which is always guaranteed to exist if the
  set is safe.  We next turn our attention to the problem of joint
  safety and stability, and give conditions under which the notions of
  control Lyapunov-barrier function (CLBF) and compatible control
  Lyapunov function (CLF) and CBF pair are guaranteed to exist.
  Finally, we identify conditions under which a CLBF and a compatible
  CLF-CBF pair can be constructed from a non-compatible CLF-CBF pair.
  Throughout the paper, we intersperse different examples and
  counterexamples to motivate our results and position them within the
  state of the art.
\end{abstract}


\section{Introduction}\label{sec:introduction}

%
%
Safety-critical control is a fundamental problem in modern
cyber-physical systems, with a rich set of applications ranging from
autonomous driving and power systems to policy design for mitigation
of epidemic spreading.  Safety refers to the ability to ensure by
design that the evolution of the dynamics stays within a desired set
of safe states. A relatively recent but promising tool to deal with
such safety specifications are control barrier functions (CBFs). An
advantage of CBFs is that they do not require computing the system’s
reachable set, which can be computationally expensive. However,
constructing CBFs is often challenging. In fact, the problem of
whether a general safe set admits a CBF has received limited
attention.  Often, CBFs are combined with control Lyapunov functions
(CLFs) to yield controllers with safety and stability guarantees.  The
problem of when such CBFs and CLFs coexist and can be combined has
also not been thoroughly studied.  Addressing these gaps is the focus
of this work.

\subsubsection*{Literature Review}

Control barrier functions (CBFs)~\cite{PW-FA:07,ADA-SC-ME-GN-KS-PT:19}
aim to render a given set forward invariant.
%
Despite their popularity and success in a variety of different
applications~\cite{SCH-XX-ADA:15,ADA-XX-JWG-PT:17,TV-SM-RH-QH:21,PM-JC:22-acc},
various fundamental questions about their properties still remain
open.  A key question is whether CBFs are guaranteed to exist for any
safe set. Several
works~\cite{SR:18,SP-AR:05,RW-CS:13,JL:22,MM-RGS:23,ADA-SC-ME-GN-KS-PT:19}
have studied such converse results, particularly for systems without
inputs, for which CBFs are referred to as \textit{barrier functions}.
The main result in~\cite{SP-AR:05} applies to continuously
differentiable vector fields without inputs and requires the existence
of a continuously differentiable function that is strictly increasing
along the solutions of the vector field.  However, such assumptions
are restrictive in many cases of interest, such as systems that admit
limit cycles. The main result in~\cite{RW-CS:13} applies to smooth
vector fields without inputs and requires the existence of a
\textit{Meyer function}~\cite
{SP-AR:05}, which is also a restrictive assumption.  The
work~\cite{SR:18} shows that, when the safe set is bounded, if the
system is robustly safe and has no inputs, then there exists a barrier
function candidate satisfying the barrier function condition strictly
at the boundary of the safe set. 

The paper~\cite{MM-MG:22} extends these converse results (also in the context 
of robust safety) to a more general class of systems and safe sets.
Finally, the recent
paper~\cite{MM-RGS:23} studies the converse safety problem extensively
in the case of systems without inputs, introducing the notion of
time-varying barrier functions. Such functions are not necessarily
smooth, and their existence is necessary and sufficient for safety.
The paper studies also the regularity properties of such time-varying
barrier functions in a variety of scenarios. Despite the
importance of this work, one must note that most of the CBF literature
is concerned with time-invariant (C)BFs, and the converse problem of
determining what is the most general set of conditions that guarantee
that a safe set admits a time-invariant (C)BF remains open.  For
systems with inputs,~\cite{ADA-SC-ME-GN-KS-PT:19} 
(which is a survey paper discussing the main technical results regarding 
CBFs, as well as some of its application domains)
shows that if the
safe set is compact and has a regular boundary, any continuously
differentiable function with the safe set as its zero superlevel set
is a~CBF.
The proof of this result relies on Nagumo's Theorem~\cite{MN:42} 
and a specific construction of the 
class $\Kc_{\infty}$ function associated with the CBF leveraging the compactness of 
the safe set.

%
%
A related problem, which has also received increasing attention, is
that of safe stabilization.  A popular approach to design safe
stabilizing controllers is to combine CBFs with
CLFs~\cite{PO-JC:19-cdc,ADA-XX-JWG-PT:17,KG-DP:21,PM-JC:23-csl}. However,
in order to have provable safety and stability guarantees, these
control design methods must ensure that the CBF and CLF are compatible
(i.e., that there exists a control input satisfying simultaneously the
associated inequalities at every point in the safe set).
Alternatively, the paper~\cite{MZR-BJ:16} proposes to unify a CLF and
a CBF into a unique function, called control Lyapunov-barrier function
(CLBF), and applies Sontag's universal formula to derive a smooth safe
stabilizing controller. However,~\cite{PB-CMK:20} points out some
limitations for the existence of such a CLBF. In a similar
vein,~\cite{MK-DK:22,MDK:23} generalize Brockett's necessary condition for
continuous state-dependent feedback stabilization~\cite{RWB:83a} in
the context of feedback stabilization and safety, which in turn
provides limitations on the existence of a CLBF or a compatible
CLF-CBF pair.
%
%
However, the converse problem, i.e., whether a CLBF or a compatible
CLF-CBF pair exists provided that safe stabilization is possible,
remains largely unexplored. A recent exception
is~\cite{YM-YL-MF-JL:22}, where converse results for safe
stabilizability are given. However, the Lyapunov-like function derived
in the main converse result~\cite[Theorem~8]{YM-YL-MF-JL:22} becomes
unbounded at the boundary of the safe set,
as a byproduct of the proof technique utilized therein.
Instead, both CLBFs and
compatible CLF-CBF pairs are bounded at the boundary of the safe
set. This distinction is relevant because the safe stabilizing
controllers based on CLBFs or compatible CLF-CBF pairs would not be
well-defined if these functions become unbounded at the boundary of
the safe set.
%
%

\subsubsection*{Statement of Contributions}
The contributions of this paper consist of converse theorems on the
existence of CBFs for the study of safety and safe stabilization of
control systems. Specifically,
\begin{enumerate}
\item We provide an example that shows that for unbounded safe sets,
  there might be candidate CBFs (i.e., functions whose zero superlevel
  set is the safe set) which are not CBFs, and candidate CBFs which
  are.  This is in contrast to the case of bounded safe sets, where
  all candidate CBFs are CBFs.  We also provide an example that shows
  that the existence of a CBF does not guarantee the existence of a
  locally Lipschitz safe feedback controller, even if the CBF
  condition is satisfied strictly at every point;
\item Given a safe set, we provide a set of general conditions on the
  dynamics and the safe set under which a CBF is guaranteed to exist.
  These conditions include safe sets for which there exists a safe
  controller such that trajectories of the closed-loop system do not
  get arbitrarily close to the boundary of the safe set, or polynomial
  systems with polynomial safe set and safe feedback.  We also define
  an extended notion of CBF, termed \textit{extended control barrier
    function} (eCBF), which relies on a generalization
  of the notion of extended class $\Kc_{\infty}$ function and show
  that they are always guaranteed to exist for any given dynamics and
  safe set;
\item Drawing on existing results in the literature, we provide a
  result that shows that if the unsafe set has a bounded connected
  component, there does not exist a CLBF or a strictly compatible
  CLF-CBF pair, and if the safe set is unbounded, there does not exist
  a CLBF. However, for a compact safe set, we show that if there
  exists a controller satisfying the CBF condition strictly and
  another controller that is stabilizing, the safe set admits a CLBF
  and a strictly compatible CLF-CBF pair.  We also show that if the
  origin is safely stabilizing, under the same conditions that we can
  guarantee the existence of a CBF, we can also guarantee the
  existence of a compatible CLF-CBF pair;
\item Finally, we show via a counterexample that the existence of a
  CLF and a CBF does not imply the existence of a strictly compatible
  CLF-CBF pair. On the positive side, we find sufficient conditions
  under which the existence of a CLF and a CBF implies the existence
  of a strictly compatible CLF-CBF pair and a compatible CLF-eCBF
  pair.
\end{enumerate}
For convenience, Table~\ref{tab:summary-results} summarizes the
  main results of the paper.

\begin{table}
  \centering
  \begin{tabular}{{ |c | c | c |}}
    \hline
    \textbf{Result} & \textbf{Assumptions} & \textbf{Statement} \\
    \hline
    Thm~\ref{thm:converse-cbf} & SC1 + safety & $\exists$ CBF \\
    \hline 
    Thm~\ref{thm:converse-ecbf} & safety & $\exists$ eCBF \\
    \hline
    Thm~\ref{thm:converse-safe-stab}~\ref{it:converse-safe-stab-clbf} & 
    \begin{tabular}{@{}c@{}} SC2 + safety \\ + stability \end{tabular} &
    \begin{tabular}{@{}c@{}} $\exists$ CLBF \\ $\exists$ strictly compatible CLF-CBF pair \end{tabular} \\
    \hline
    Thm~\ref{thm:converse-safe-stab}~\ref{it:converse-safe-stab-clf-cbf} &
    \begin{tabular}{@{}c@{}} SC1 \\ + safe stabilization \end{tabular} & $\exists$ compatible CLF-CBF pair \\
    \hline
    Thm~\ref{thm:converse-safe-stab}~\ref{it:converse-safe-stab-clf-ecbf} & 
    safe stabilization &
    $\exists$ compatible CLF-eCBF pair \\
    \hline
    Prop~\ref{prop:compatible-clf-ecbf} & 
    \begin{tabular}{@{}c@{}} SC3 + safety \\ + stability \end{tabular} & $\exists$ compatible CLF-eCBF pair \\
    \hline
  \end{tabular}
  \caption{ Summary of the main results in this
      paper. SC1 stands for the different sufficient conditions
      outlined in Theorem~\ref{thm:converse-cbf}, SC2 stands for the
      sufficient conditions in
      Theorem~\ref{thm:converse-safe-stab}~\ref{it:converse-safe-stab-clbf},
      and SC3 for the sufficient conditions in
      Proposition~\ref{prop:compatible-clf-ecbf}. }
  \label{tab:summary-results}
\end{table}

\section{Preliminaries}\label{sec:preliminaries}

In this section we introduce some notation and preliminaries on CLFs,
CBFs, and CLBFs.

\subsection{Notation}
We denote by $\mathbb{Z}_{>0}$, $\real$, $\real_{\geq0}$, and
$\real_{<0}$ the set of positive integers, real, nonnegative real
numbers, and negative real numbers, resp. We write $\Int(\Sc)$,
$\partial\Sc$, $\clos(\Sc)$ for the interior, boundary, and closure of
the set $\Sc$, resp. 
The symbol $\mathbf{0}_n$ stands for the $n$-the dimensional zero vector.
Given $x\in\real^{n}$, $\norm{x}$ denotes its
Euclidean norm.  Given two functions $b_1:\real\to\real$,
$b_2:\real\to\real$, we say that $b_1(x)=O(b_2(x))$ near some real
number $a$ if there exist positive numbers $\delta$ and $M$ such that
$|b_1(x)|\leq M b_2(x)$ for all $x$ with $0<\norm{x-a}<\delta$.

Given $f:\real^{n}\to\real^{n}$, $g:\real^{n}\to\real^{n\times m}$ and
a smooth function $W:\real^{n}\to\real$, the notation
$L_{f}W:\real^{n}\to\real$ (resp. $L_gW:\real^n\to\real^m$) denotes
the Lie derivative of $W$ with respect to $f$ (resp. $g$), that is
$L_{f}W=\nabla W^T f$ (resp. $\nabla W^Tg$). We denote by
$\Cc^1(\real^n)$ and $\Cc^\infty(\real^n)$ the set of continuously
differentiable and infinitely continuously differentiable functions in
$\real^{n}$, resp. Given a multi-index
$\rho=(\rho_1,\rho_2,\hdots,\rho_n)$ and a smooth scalar function $f$,
$D^{\rho}f(x) = \frac{\partial^{\rho_1}}{\partial x_1}
\frac{\partial^{\rho_2}}{\partial x_2} \hdots
\frac{\partial^{\rho_n}}{\partial x_n} f(x)$. We denote by $|\rho|$
the sum of the components of $\rho$.  Given a differentiable function
$h:\Dc\to\real$ and $\lambda\in\real$, $\lambda$ is a regular value of
$h$ if $\pder{h}{x}(x)\neq0$ for all
$x\in\setdef{y\in\Dc}{h(y)=\lambda}$. A function $V:\real^n\to\real$
is positive definite if $V(0)=0$ and $V(x)>0$ for all $x\neq0$. $V$ is
proper in a set $\Gamma$ if the set $\setdef{x\in\Gamma}{V(x)\leq c}$
is compact for any $c\geq0$. $V$ is proper if it is proper in its
domain.  A continuous function
$\alpha:\real_{\geq0} \to \real_{\geq0}$ is of class $\Kc$ if it is
strictly increasing and $\alpha(0)=0$ and of class $\Kc_{\infty}$ if
additionally $\lim\limits_{r\to\infty}\alpha(r)=\infty$.  If
$\alpha:\real\to\real$ is strictly increasing, $\alpha(0)=0$, and
$\lim\limits_{r\to\pm\infty}\alpha(r)=\pm\infty$, we say that $\alpha$
is of extended class $\Kc_{\infty}$.

A closed set $\Cc$ is forward invariant under the dynamical system
$\dot{x}=f(x)$ if any trajectory with initial condition in $\Cc$
at time $t=0$ remains in $\Cc$ for all times $t\geq0$.  A closed set $\Cc$ is safe
for the control system $\dot{x}=f(x,u)$, with
$f:\real^n\times\real^m\to\real^n$ locally Lipschitz, if there exists
a locally Lipschitz control $u_{\text{sf}}:\real^n\to\real^m$ such that $\Cc$ is
forward invariant for $\dot{x}=f(x,u_{\text{sf}}(x))$.
Furthermore, we call $u_{\text{sf}}$ a safe controller in~$\Cc$.
A point $p$ is safely stabilizable on a
set $\Cc$ if there exists a locally Lipschitz control 
$u_{\text{sf}}:\real^{n}\to\real^{m}$ such
that $\Cc$ is forward invariant for $\dot{x}=f(x,u_{\text{sf}}(x))$ and $p$ is
asymptotically stable with region of attraction containing $\Cc$. Such
a controller is a \textit{safe stabilizing} controller in~$\Cc$.

\subsection{Control Lyapunov and Barrier Functions}

In this section we introduce the notions of control Lyapunov and
barrier functions. Throughout the paper we consider the nonlinear
control system
\begin{align}\label{eq:control-sys}
  \dot{x}=f(x,u),
\end{align}
where $f:\real^{n}\times\real^{m}\to\real^{n}$ is locally Lipschitz,
with $x\in\real^n$ the state and $u\in\real^{m}$ the input.

\begin{definition}\longthmtitle{Control Lyapunov
    Function~\cite{EDS:98,RAF-PVK:96a}}\label{def:clf}
  Given a set $\Gamma\subseteq\real^{n}$, with
  $\mathbf{0}_n\in\Gamma$, a continuously differentiable function
  $V:\Gamma\to\real$ is a \textbf{CLF} on $\Gamma$ for the
  system~\eqref{eq:control-sys} if it is proper in $\Gamma$, positive
  definite, and there exists a positive definite function
  $W:\Gamma\to\real$ such that, for each
  $x\in\Gamma\backslash \{0\}$, there exists a control $u\in\real^{m}$
  satisfying
  \begin{align}\label{eq:clf-ineq}
    \nabla V(x)^T f(x,u) \leq -W(x).
  \end{align}
\end{definition}
\medskip
A Lipschitz controller $k:\real^{n}\to\real^{m}$ such that $u=k(x)$
satisfies~\eqref{eq:clf-ineq} for all $x\in\Gamma\backslash\{0\}$ makes
the origin of the closed-loop system asymptotically stable. Hence,
CLFs provide a way to guarantee asymptotic stability.

Next we recall the notion of control barrier function
(CBF)~\cite{ADA-SC-ME-GN-KS-PT:19}.  Consider the set $\Cc$ defined as
the zero-superlevel set of a continuously differentiable
function $h:\real^{n}\to\real$ as follows
\begin{subequations}
  \begin{align}
    \Cc&=\setdef{x\in\real^n}{h(x)\geq0},
    \\
    \partial\Cc&=\setdef{x\in\real^n}{h(x)=0},
    \\
    \text{Int}(\Cc)&=\setdef{x\in\real^n}{h(x)>0}.
  \end{align}
  \label{eq:safe-set}
\end{subequations}
Further assume that $\text{Int}(\Cc) \neq \emptyset$.
A continuously differentiable 
function $h:\real^n\to\real$ satisfying~\eqref{eq:safe-set} is referred to 
as a \textit{candidate CBF} of $\Cc$.

\begin{definition}\longthmtitle{Control Barrier
    Function}\label{def:cbf}
  Let $h:\real^{n}\to\real$ be a continuously differentiable function
  satisfying~\eqref{eq:safe-set}. The
  function $h$ is a \textbf{CBF}
  of $\Cc$ for the
  system~\eqref{eq:control-sys} if there exists an extended class
  $\mathcal{K}_\infty$ function $\alpha$ such that, for all
  $x\in \Cc$, there exists a control $u\in\real^{m}$ satisfying
  \begin{align}\label{eq:cbf-ineq}
      \nabla h(x)^T f(x,u)+\alpha(h(x))\geq0.
  \end{align}
\end{definition}
\smallskip

In this paper we stick with Definition~\ref{def:cbf}, which is widely
used in the CBF literature, as opposed to the notion of time-varying
barrier function in~\cite[Definition 15]{MM-RGS:23}.  The following
result establishes that the existence of a CBF of~$\Cc$ certifies its
safety.

\begin{theorem}\longthmtitle{CBFs
    certify
    safety~\cite[Theorem~2]{ADA-SC-ME-GN-KS-PT:19}}\label{thm:cbfs-certify-safety}
  Let $\Cc\subset\real^{n}$, $h$ be a CBF of $\Cc$ for the
  system~\eqref{eq:control-sys}, and $0$ be a regular value
  of~$h$. Any Lipschitz controller $u_{\text{sf}}:\real^{n}\to\real^{m}$ that
  satisfies
  \begin{align}
    u_{\text{sf}}(x)\in K_{\cbf}(x)\!=\!\setdef{u\in \real^m}{\nabla h(x)^T
    f(x,u)\geq-\alpha(h(x))}
  \end{align}
  for all $x\in\Cc$ renders the set $\Cc$ forward invariant.
\end{theorem}

The following result states that for compact safe sets the converse of Theorem~\ref{thm:cbfs-certify-safety}
also holds.

\begin{theorem}\longthmtitle{Converse CBF result for compact safe
    sets~\cite[Theorem
    3]{ADA-SC-ME-GN-KS-PT:19}}\label{thm:compact-converse-cbf}
  Let $\Cc$ be a compact set defined as in~\eqref{eq:safe-set} and
  assume that $0$ is a regular value of $h$. If $\Cc$ is safe for
  system~\eqref{eq:control-sys}, then $h\vert_{\Cc}:\Cc\to\real$ is a
  CBF of $\Cc$.
\end{theorem}

Note that the result not only states the existence of a CBF but also
that the function defining the set $\Cc$ is itself a CBF.  Next we
comment on the assumptions of Theorems~\ref{thm:cbfs-certify-safety}
and~\ref{thm:compact-converse-cbf} and establish different connections
with the literature.

\begin{remark}\longthmtitle{Existence of Lipschitz safe
    controllers}\label{rem:existence-locally-Lipschitz-safe-controllers}
  {\rm
  As pointed out in Theorem~\ref{thm:cbfs-certify-safety}, any locally
  Lipschitz safe controller satisfying that CBF condition renders the
  set safe.  However, in general, such locally Lipschitz controller
  might not exist even if a CBF is available.  In fact,~\cite[Example
  III.5]{MA-NA-JC:25-tac} shows that the so-called
  \textit{minimum-norm} controller (obtained at every $x\in\real^n$ as
  the controller with smallest norm that
  satisfies~\eqref{eq:cbf-ineq}) can be unbounded.  Since the
  minimum-norm controller is unbounded, this example shows that even
  if a CBF exists, there might not exist a locally Lipschitz
  controller satisfying~\eqref{eq:cbf-ineq}.  However,~\cite[Lemma
  III.2]{MA-NA-JC:25-tac} shows that if~\eqref{eq:control-sys} is
  control-affine, $\Cc$ is compact and the CBF
  condition~\eqref{eq:cbf-ineq} holds strictly at the boundary, then
  the minimum-norm controller is locally Lipschitz.  
  Alternatively, if $\Cc$ is compact and there exists an open set $\Dc$ containing $\Cc$ for which 
  the CBF condition is feasible, then a Lipschitz safe controller also exists~\cite[Theorem 5]{RK-ADA-SC:21}.
  We next
  complement this discussion by presenting an example inspired
  by~\cite{EDS-HJS:80} that shows that if the system is not
  control-affine, even if $\Cc$ is compact and the CBF
  condition~\eqref{eq:cbf-ineq} holds strictly at the boundary, there
  might not exist a continuous safe controller.  Let
  $h:\real^2\to\real$ be defined as $h(x,y) = -(x^2 + y^2) + 10$ and
  take $\Cc$ as in~\eqref{eq:safe-set}.  Consider the system
  \begin{subequations}\label{eq:existence-cbf-no-existence-continuous-controller} 
    \begin{align}
      \dot{x}
      &=  x \big( (u-1)^2 - (x-1) \big) \big( (u+1)^2 + (x-2)
        \big),
      \\
      \dot{y} &= y \big( (u-1)^2 - (x-1) \big) \big( (u+1)^2 + (x-2)
      \big).
    \end{align}
  \end{subequations}
  Let us show that $h$ is a CBF of $\Cc$ for
  system~\eqref{eq:existence-cbf-no-existence-continuous-controller}.
  Take $(x,y)\in\Cc$, and note that~\eqref{eq:cbf-ineq} is equivalent
  to
  \begin{align}\label{eq:cbf-ineq-counterexample-nonsmooth-feedback}
    \notag
    -2(x^2+y^2) \big( (u - 1)^2 - (x-1) \big) \big( (u + 1)^2 + & (x - 2) \big) \! \\
    & \geq \! -\alpha(h(x,y)).
  \end{align}
  Note that the set of points $(x,u)\!\in\!\real^2$ that satisfy
  \begin{align}\label{eq:auxx}
    \big( (u-1)^2 \! - \! (x-1)
    \big) \big( (u+1)^2 \! + \! (x-2) \big) \! \leq \! 0 ,
  \end{align}
  consists of two disjoint connected sets, one formed by the points to
  the right of the parabola $x=1+(u-1)^2$, and the other formed by the
  set of points to the left of the parabola $x=2-(u+1)^2$.  
  This set is illustrated in Figure~\ref{fig:discontinuous}.
  \begin{figure}[htb]
    \centering
    \includegraphics[width=0.45\textwidth]{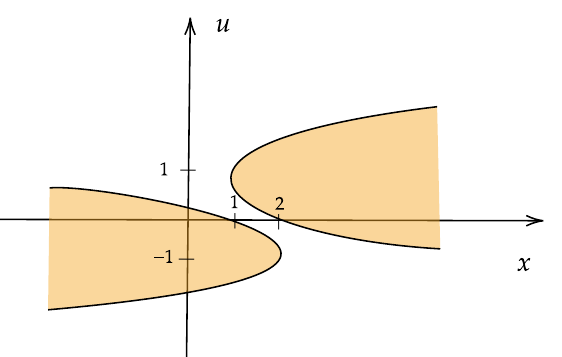}
    \caption{Illustration of the set of points $(x,u)$
      satisfying~\eqref{eq:auxx}. For every value of $x$, there are
      controls (colored in orange) that
      satisfy~\eqref{eq:auxx}.}\label{fig:discontinuous}
  \end{figure}
  Note that the projection of these two sets onto the $x$ axis covers
  the whole axis. Hence, $h$ is a CBF.  Now let 
  $\tilde{u}_0:\Cc\to\real$ be a
  controller such that $u=\tilde{u}_0(x,y)$ satisfies~\eqref{eq:cbf-ineq-counterexample-nonsmooth-feedback} 
  at all points in $\Cc$, and let
  $\tilde{u}:[-\sqrt{10},\sqrt{10}]\to\real$ be defined as
  $\tilde{u}(x) = \tilde{u}_0(x,\sqrt{10-x^2})$.  Note that since the two
  connected sets are disjoint, $\tilde{u}$ can not be continuous.  This
  implies that $\tilde{u}_0$ can also not be continuous.  We note also that in
  this example, $\Cc$ is compact and for all
  $x\in\real$, there exists $u\in\real$
  satisfying the inequality
  $\big( (u-1)^2 \! - \! (x-1) \big) \big( (u+1)^2 \! + \! (x-2)
  \big)\leq 0$ strictly, which means that for all $(x,y)\in\partial\Cc$,
  there exists $u\in\real$ satisfying~\eqref{eq:cbf-ineq-counterexample-nonsmooth-feedback} 
  strictly, and there also exists a neighborhood of $\Cc$ for which~\eqref{eq:cbf-ineq-counterexample-nonsmooth-feedback}
  is feasible.
  Hence, the only hypothesis
  of~\cite[Lemma~III.2]{MA-NA-JC:25-tac} and 
  ~\cite[Theorem 5]{RK-ADA-SC:21}
  that fails is that the system is not control-affine.  \demo}
\end{remark}

\begin{remark}\longthmtitle{Minimal Control Barrier Functions}
  {\rm
  The work~\cite{RK-ADA-SC:21} shows that the regularity assumption
  can be dropped in both Theorems~\ref{thm:cbfs-certify-safety}
  and~\ref{thm:compact-converse-cbf}.  This work identifies the
  minimal set of conditions that guarantee safety and defines the set
  of functions satisfying these conditions as \textit{minimal
    (control) barrier functions} (M(C)BFs).
  %
  Even though here we focus on CBFs, we also point out various
  connections of our results to M(C)BFs.  \demo}
\end{remark}

When dealing with both stability and safety requirements under the
dynamics~\eqref{eq:control-sys}, it is important to note that an input
$u$ might satisfy~\eqref{eq:clf-ineq} but not~\eqref{eq:cbf-ineq}, or
vice versa. The following definition captures when the CLF and the CBF
are compatible.

\begin{definition}\longthmtitle{Compatibility of CLF-CBF
    pair~\cite[Definition 3]{PM-JC:23-csl}}\label{def:compat-clf-cbf}
  Let $\Gamma \subseteq \real^{n}$ be open, $\Cc \subset \Gamma$
  closed, $V$ a CLF on $\Gamma$ and $h$ a CBF of~$\Cc$. Then, $V$ and
  $h$ are (strictly) \textbf{compatible} at $x \in \Cc$ if there
  exists $u\in\real^{m}$ satisfying~\eqref{eq:clf-ineq}
  and~\eqref{eq:cbf-ineq} (strictly) simultaneously.
\end{definition}

We refer to $V,h$ as a compatible pair in $\Cc$ if they are compatible
at every point in $\Cc$. Similarly, $V,h$ are a strictly compatible
pair in $\Cc$ if they are strictly compatible at every point in
$\Cc\backslash\{0\}$. 
For control-affine systems, the work~\cite{PO-JC:19-cdc} shows that if $V$
and $h$ are strictly compatible in $\Cc$, then there exists a smooth
safe stabilizing controller. If $V$ and $h$ are only
compatible in $\Cc$, the existence of a smooth safe stabilizing
controller is not guaranteed in general. However,~\cite{PM-JC:23-csl}
provides additional technical conditions under which the control
design considered therein is locally Lipschitz and achieves safe
stabilization. If instead the CBF and CLF inequalities are included as
hard constraints of an optimization problem, one can use the theory of
parametric optimization to obtain conditions under which the resulting
optimization-based controller satisfies desirable regularity properties without
requiring strict compatibility, cf.~\cite{PM-AA-JC:25-ejc}.

\subsection{Control Lyapunov-Barrier Functions}

Here we recall the notion of Control Lyapunov-Barrier Function (CLBF)
introduced in~\cite{MZR-BJ:16} to design safe stabilizing
controllers. Although the original definition is for control-affine
systems, here we present it for general control
systems~\eqref{eq:control-sys}.

\begin{definition}\longthmtitle{Control Lyapunov-Barrier
    Function~\cite[Definition 2]{MZR-BJ:16}}\label{def:clbf}
  A proper and lower-bounded function $\bar{V}\in\Cc^{1}(\real^n)$ is a
  Control Lyapunov-Barrier Function (\textbf{CLBF}) of
  $\real^n\backslash\Cc$ if it satisfies
  \begin{subequations}
    \begin{align}
      & \bar{V}(x)>0 \quad \forall x \in \real^n\backslash\Cc
        , \label{eq:positive} 
      \\
      &\Uc :=
        \setdef{x\in\real^n}{\bar{V}(x)\leq0}\neq\emptyset, \label{eq:nonempty-sublevel-set}  
      \\
      &
        ( \overline{\Cc\backslash\Uc} ) \cap
        ( \overline{\real^n\backslash\Cc}
        )=\emptyset, \label{eq:empty-intersection} 
      \\
      & \inf_{u\in\real^m} \nabla \bar{V}(x)^T f(x,u)<0 \quad
        \forall
        x\in \Cc\backslash\{0\}. \label{eq:clbf-negative-derivative} 
    \end{align}
    \label{eq:clbf}
  \end{subequations}
\end{definition}

In the case where $\Uc=\Cc$, the
conditions~\eqref{eq:positive},~\eqref{eq:nonempty-sublevel-set} are
reminiscent of~\eqref{eq:safe-set} (with the sign changed),
and~\eqref{eq:clbf-negative-derivative} is a more restrictive version
of the inequality in Definition~\ref{def:cbf}.  On the other hand, the
requirement that $\bar{V}$ is proper
and~\eqref{eq:clbf-negative-derivative} resemble the definition of CLF
(cf. Definition~\ref{def:clf}), although in this case we do not
require $\bar{V}$ to be positive definite.  Finally,
condition~\eqref{eq:empty-intersection} is technical and guarantees
that trajectories never enter the unsafe set even if their initial
value of $\bar{V}$ is positive (cf.~\cite[Theorem 2]{MZR-BJ:16}).
This condition is trivially satisfied in the case~$\Uc=\Cc$. Given a
CLBF,~\cite[Proposition 3]{MZR-BJ:16} provides an explicit
construction of a safe stabilizing controller in $\Cc$.
In particular, this means that if there exists a CLBF of $\real^n\backslash\Cc$, 
then $\Cc$ is safe.


\section{Problem Statement}\label{sec:problem-statement}
We consider a control system of the form~\eqref{eq:control-sys} and a
safe set $\Cc$ described by a differentiable function $h$ as
in~\eqref{eq:safe-set}. We are broadly motivated by questions about
the existence of functions certifying stability and
safety. Specifically, our goal is to answer the following questions:
\begin{itemize}
\item[(P1)] If $\Cc$ is safe for the system, does it always admit a
  CBF? This problem corresponds to the converse of
  Theorem~\ref{thm:cbfs-certify-safety}, to which
  Theorem~\ref{thm:compact-converse-cbf} provides an answer in case
  $\Cc$ is compact. Here we intend to establish a more general result;

\item[(P2)] Under what conditions can the existence of a CLBF or a
  (strictly) compatible CLF-CBF pair be guaranteed? These problems are
  motivated by the fact that in either case feedback controllers that
  achieve safe stabilization can be designed under appropriate
  technical conditions, as described in
  Section~\ref{sec:preliminaries}.

\item[(P3)] Does the existence of a (not necessarily compatible)
  CLF-CBF pair imply the existence of a (strictly) compatible CLF-CBF
  pair? This question shares its motivation with (P2) and the ease
  afforded by identifying a CLF and a CBF independently of each other.

\end{itemize}

We address these problems in the remainder of the paper,
providing sufficient conditions under which each of them can be
solved.

\section{Converse Results for Safety}\label{sec:converse-sf}

In this section, we address problem~(P1). Note that
Theorem~\ref{thm:compact-converse-cbf} already provides a partial
answer for the case when $\Cc$ is compact. The treatment of this
section establishes more general results.  We start with an example
showing that, given an arbitrary safe set $\Cc$, not every candidate
CBF of $\Cc$ is a CBF of $\Cc$.

\begin{example}\longthmtitle{Choice of candidate CBF matters for
    unbounded safe sets}\label{ex:wrong-candidate}
  {\rm 
Consider the function $h:\real^{2}\to\real$ given by $h(x,y)=x$ and
the set $\Cc$ defined as in~\eqref{eq:safe-set}. Consider the system
  \begin{align*}
    \dot{x}&=xy+1,
    \\
    \dot{y}&=-y+u.
  \end{align*}
  Since $\frac{d}{dt}h(x,y)=1$ when $x=0$ for any choice of $u$, by Nagumo's
  Theorem~\cite{MN:42}, $\Cc$ is safe. However, $h$ is not a CBF. To
  show this, assume that it is. Therefore, there exists an extended class
  $\Kc_{\infty}$ function $\alpha$ satisfying
  \begin{align*}
    \nabla h(x,y)^T
    \begin{pmatrix}
      xy
      \\
      -y+u
    \end{pmatrix}
    = xy+1 \geq -\alpha(x).
  \end{align*}
  Note that, for any fixed $x>0$, $\alpha(x)$ is constant, but by
  taking $y$ sufficiently negative, $xy$ can be arbitrarily negative,
  and the inequality $xy+1 \geq -\alpha(x)$ will not be satisfied.
  Since this argument holds for any extended class $\Kc_\infty$ function
  $\alpha$, $h$ is not a CBF. 
  Similarly, by assuming that $\alpha$ is a minimal function~\cite[Definition 1]{RK-ADA-SC:21}
  also shows that $h$ is not a \textit{minimal control barrier function}~\cite[Definition 3]{RK-ADA-SC:21}.
  
  However, one can show that the function
  $\tilde{h}(x,y)=e^{y}x$, which also satisfies~\eqref{eq:safe-set},
  is in fact a CBF. Indeed, note that
  \begin{align*}
    \frac{d}{dt}\tilde{h}(x,y)=\nabla \tilde{h}(x,y)^T
    \begin{pmatrix}
      xy+1
      \\
      -y+u
    \end{pmatrix}
    = e^y+e^y xu .
  \end{align*}
  Taking $u=0$ for all $(x,y)\in\real^2$ makes $\tilde{h}$ satisfy~\eqref{eq:cbf-ineq}
  for any class $\Kc_{\infty}$ function~$\alpha$.  \problemfinal
  }
\end{example}

The relevance of this example is in showing that the assumption that
$\Cc$ be compact is critical for
Theorem~\ref{thm:compact-converse-cbf} (as well
as~\cite[Corollary~2]{RK-ADA-SC:21}) to hold, since this result
establishes that any candidate CBF of $\Cc$ is a CBF of $\Cc$. The
extension of this result to unbounded safe sets therefore requires
adjustments in the technical approach. Interestingly,
  we should point out that some aspects of Lyapunov theory for
  stability are also not fully understood in the case of unbounded
  attractors~\cite{BL-WY-MC:22}.

\begin{remark}\longthmtitle{Other counterexamples in the literature}
  {\rm
  We explain here the relative value and qualitative differences of
  Example~\ref{ex:wrong-candidate} with respect to other
  counterexamples in the literature.~\cite[Remark 8]{ADA-XX-JWG-PT:17}
  gives an example of a safe set for which a differentiable function
  satisfying~\eqref{eq:safe-set} is not a CBF, but does not specify
  whether there exists another function with the same properties that
  is.~\cite[Example 1]{MM-RGS:23} provides an example of a safe set
  with empty interior which does not admit a continuous barrier
  certificate~\cite{PW-FA:07} that is only a function of the
  state. However, the system considered does not have control inputs
  and the notion of barrier certificate is different from the standard
  notion of CBF considered here (for which, for instance, safe sets
  have non empty interior).  Finally, the counterexample
  in~\cite[Example 5]{MM-RGS:23} defines a safe set that is not
  expressible as the superlevel set of a differentiable function.
  \demo}
\end{remark}



\subsection{Converse Theorem for CBFs}

Example~\ref{ex:wrong-candidate} shows that, for an arbitrary safe set
$\Cc$, not every function satisfying~\eqref{eq:safe-set} is a CBF, and
in turn also raises the question of whether a CBF might even
exist. The following result states conditions under which this is the
case.

\begin{theorem}\longthmtitle{Converse CBF result for arbitrary
    sets}\label{thm:converse-cbf}
  Given a control system~\eqref{eq:control-sys}, let $\Cc$ be a set
  for which there exists a continuously differentiable function
  $h:\real^n\to\real$ satisfying~\eqref{eq:safe-set}. Suppose that
  $\Cc$ is safe and any of the following assumptions hold:
  \begin{enumerate}
  \item\label{it:minus1}
      there exists an extended class
      $\Kc_{\infty}$ function $\alpha$ and a function
      $u_*:\real^n\to\real^m$ such that, for all $r\geq0$,
      \begin{align}\label{eq:cbf-condition-with-inf}
        \inf\limits_{ \setdef{x\in\real^n}{h(x)\in[0,r]} } \nabla
        h(x)^T f(x,u_*(x)) \geq -\alpha(r); 
      \end{align}
  \item\label{it:zero} there exists a locally Lipschitz
      safe controller $u_0:\real^n\to \real^m$ and a positive
      function $\nu:\text{Int}(\Cc)\to\real_{>0}$ such that, for any
      $x_0\in\text{Int}(\Cc)$, the trajectory $x(\cdot)$ of
      $\dot{x}=f(x,u_0(x))$ with initial condition at $x_0$, satisfies
      $h(x(t)) \geq \nu(x_0)>0$ for all $t\geq0$;
    %
  \item\label{it:second} the function $f$ is continuously differentiable, there
    exists a continuously differentiable safe controller $\hat{u}:\real^{n}\to\real^{m}$,
    %
    %
    positive integers $M_{2}\in\mathbb{Z}_{>0}$, $N_{2}\in\mathbb{Z}_{>0}$,
    and positive constants
    $\{b_j\}_{j\in\{1,\hdots,M_2\}}$, and
    $\{c_{k}\}_{k\in\{1,\hdots,N_2\}}$ such that
    \begin{subequations}
      \begin{align}
        \norm{\nabla(\norm{f(x,\hat{u}(x))}^2)}^2
        &\leq\sum_{j=0}^{M_2} c_j
          \norm{f(x,\hat{u}(x))}^j,
        \\ 
        \norm{\nabla h(x)}^2
        &
          \leq \sum_{k=0}^{N_2} b_k \norm{f(x,\hat{u}(x))}^k,\label{eq:grad-boundedness-h}
      \end{align}
      \label{eq:grad-boundedness}
    \end{subequations}
    for all $x\in\Cc$;
  \item\label{it:third} the set $\Cc$ is compact.
  \end{enumerate}
  Then, there exists a CBF of $\Cc$.
\end{theorem}
\begin{proof}
  Note that~\ref{it:third} is simply
  Theorem~\ref{thm:compact-converse-cbf}.  

  \textbf{To show~\ref{it:minus1},} note that
  if~\eqref{eq:cbf-condition-with-inf} holds, then for any $x\in\real^n$,
  we have $\nabla h(x)^T f(x,u_*(x)) \geq -\alpha(h(x))$, and hence $h$ is a CBF of $\Cc$.

  \textbf{We now prove \textbf{\ref{it:zero}} and divide the proof in two steps.}  
  First,
  we construct a function that is differentiable almost everywhere,
  satisfies~\eqref{eq:safe-set} and for which the CBF
  condition~\eqref{eq:cbf-ineq} holds at all points where it is
  differentiable. Second, we smoothen this function and obtain an
  actual CBF.

  \underline{First step: construction of a \textit{CBF almost
      everywhere}.} For this step, we rely on the
    techniques in~\cite{JJC-DL-KS-CJT-SLH:21}, which studies the
    connection between Hamilton-Jacobi reachability and CBFs.  Let us
  consider the cost function $V:\real^n\times\real_{<0}\to\real$
  \begin{align*}
    V(x,t)=\min_{s\in[t,0]} h(x(s)),
  \end{align*}
  which captures the minimum value of $h$ along the trajectory
  $x(\cdot)$ that solves~\eqref{eq:control-sys}, with initial
  condition $x$, initial time $t<0$ and control $u_0$ (note that as
  opposed to~\cite{JJC-DL-KS-CJT-SLH:21}, here we omit the
  maximization over all possible controllers and simply use $u_0$).
  As explained in~\cite[Section II.D]{JJC-DL-KS-CJT-SLH:21}, by
  extending the definition of $V$ for infinite time as
  $V_{-\infty}(x):=\lim_{t\to-\infty}V(x,t)$,
  %
  %
  we obtain a time-invariant function whose zero-superlevel set is the
  largest forward invariant set of $\dot{x}=f(x,u_0(x))$ contained in
  $\Cc$.  In our case, since $u_0$ is a safe controller,
  $\Cc=\setdef{x\in\real^n}{V_{-\infty}(x)\geq0}$.
    Moreover, since $u_0$ is such that all trajectories of
    $\dot{x}=f(x,u_0(x))$ with initial condition
    $x_0\in\text{Int}(\Cc)$ satisfy $h( x(t) ) \geq \nu(x_0)$ for all
    $t\geq0$, it follows that $V(x_0,t) \geq \nu(x_0)$ for all $t<0$
    and therefore $V_{-\infty}(x_0) \geq \nu(x_0) > 0$ for all
    $x_0\in\text{Int}(\Cc)$.
  As also noted in~\cite[Section
    II.D]{JJC-DL-KS-CJT-SLH:21}, for all points in $\Cc$ where the
  gradient of $V_{-\infty}$ exists,
  \begin{align}\label{eq:cbf-vinf}
    \nabla V_{-\infty}(x)^T f(x,u_0(x))\geq-\alpha(V_{-\infty}(x)),
  \end{align}
  for any smooth extended class $\Kc_{\infty}$ function
  $\alpha$. Since
  $V_{-\infty}$ might not be differentiable at some points, it might
  not be a valid CBF. However,
  since $f$, $u_0$ and $h$ are locally Lipschitz,
  $V_{-\infty}$ is locally Lipschitz (cf.~\cite[Remark 1]{JJC-DL-KS-CJT-SLH:21}), 
  and therefore by Rademacher's Theorem $V_{-\infty}$ is differentiable
  almost everywhere~\cite[Theorem 7.11]{WR:87}.
  
  \underline{Second step: smoothing.}  The rest of the proof smoothens
  $V_{-\infty}$ to obtain a valid CBF. To do so, we
    follow a procedure closely related to that of~\cite{YL-EDS-YW:96}
    for smoothing Lyapunov functions.
  Let us start by showing
  that we can smoothen $V_{-\infty}$ at the interior of $\Cc$ and
  guarantee~\eqref{eq:cbf-vinf} for all $x\in\Int(\Cc)$ for the
  smoothened version of $V_{-\infty}$. Indeed, by~\cite[Theorem
  B.I]{YL-EDS-YW:96} 
  there exists a smooth
  $\Psi:\text{Int}(\Cc)\to\real$ such that for all $x\in\Int(\Cc)$,
  \begin{align*}
    |V_{-\infty}(x)-\Psi(x)|<\min\{\frac{1}{2}V_{-\infty}(x),1\},
    \\
    \nabla \Psi(x)^T f(x,u_0(x)) \geq - 2 \alpha(V_{-\infty}(x)).
  \end{align*}
  Since $V_{-\infty}(x)>0$ for all $x\in\Int(\Cc)$,
  then $\Psi(x)>V_{-\infty}(x)-\frac{1}{2}V_{-\infty}(x)=\frac{1}{2}V_{-\infty}(x)>0$
  for all $x\in\Int(\Cc)$.  Now extend $\Psi$ at $\partial\Cc$ so that
  $\Psi(x)=0$ for all $x\in\partial\Cc$.  It follows that $\Psi$
  defined in this way is smooth in $\Int(\Cc)$ and continuous in
  $\Cc$.  Moreover, since $\alpha$ is increasing,
  $2\alpha(V_{-\infty}(x))\leq 2\alpha(2\Psi(x))$. Hence, by defining
  $\bar{\alpha}(r)=2\alpha(2r)$, $\bar{\alpha}$ is smooth, extended
  class $\Kc_{\infty}$ and for all
  $x\in\Int(\Cc)$ it holds that
  $\nabla \Psi(x)^T f(x,u_0(x)) \geq -\bar{\alpha}(\Psi(x))$.
  Now, extend $\Psi$ in
  $\real^n\backslash\Cc$ in such a way that $\Psi$ is smooth in
  $\real^{n}\backslash\Cc$ and $\Psi(x)<0$ for all
  $x\in\real^{n}\backslash\Cc$ so that $\Psi$ is continuous in
  $\real^{n}$ and smooth in $\real^{n}\backslash\partial\Cc$.
  Let us now use $\Psi$ to construct a function $\Phi:\real^n\to\real$
  that is smooth in all of $\real^n$.  In order to do so, let us show
  that there exists an extended class $\Kc_{\infty}$ function $\beta$
  such that $\Phi :=\beta \circ \Psi$ is smooth in all of $\real^{n}$.
  %
  %
  The proof follows that of~\cite[Lemma 4.3]{YL-EDS-YW:96},
  where the main difference is that in our case $\Psi$ takes positive
  and negative values.  For $i\in\integerspos$, let $K_{i}$ be compact
  subsets of $\real^{n}$ such that
  $\partial\Cc\subset \bigcup_{i=1}^{\infty} K_{i}$. For any
  $k\in\integerspos$, let:
  \begin{align*}
    I_k^+ := \left( \frac{1}{k+2}, \frac{1}{k} \right), \quad I_k^- :=
    \left(-\frac{1}{k}, -\frac{1}{k+2} \right). 
  \end{align*}
  Pick also smooth $\C^{\infty}(\real)$ functions
  $\gamma_{k}^+:\real\to[0,1]$, $\gamma_{k}^-:\real\to[-1,0]$
  satisfying
  \begin{itemize}
  \item $\gamma_{k}^+(t)=0$ if $t\notin I_k^+$,
  \item $\gamma_{k}^+(t)>0$ if $t\in I_k^+$,
  \item $\gamma_{k}^-(t)=0$ if $t\notin I_k^-$,
  \item $\gamma_{k}^-(t)<0$ if $t\in I_k^-$.
  \end{itemize}
  Define also for any $k\in\integerspos$,
  \begin{align*}
    \mathcal{G}_k := \setdef{x\in\real^n}{x\in \bigcup_{i=1}^k K_i,
    \quad \Psi(x)\in \clos(I_k^+ \cup I_k^-)}. 
  \end{align*}
  Observe that $\mathcal{G}_{k}$ is compact
  for all $k\in\integerspos$ (because the sets $K_i$ are compact and 
  $\Psi$ is continuous)
  and hence for each $k\in\integerspos$ there exists $c_{k}\in\real$
  satisfying:
  \begin{enumerate}
  \item $c_{k}\geq1$,
  \item $c_{k}\geq |(D^\rho \Psi)(x)|$ for any multi-index
    $|\rho|\leq k$ and $x\in\mathcal{G}_{k}$,
  \item $c_{k}\geq |\gamma_k^{(i)}(t)|$ for any $i\leq k$ and any
    $t\in\real_{>0}$.
  \end{enumerate}
  Choose the sequence $\{d_{k} \}_{k\in\integerspos}$ to satisfy
  \begin{align*}
    0<d_k<\frac{1}{2^k (k+1)! c_k^k}, \quad k\in\integerspos.
  \end{align*}
  Now, define
  \begin{align*}
    \gamma(t)=\sum_{k=1}^\infty d_k (\gamma_k^+(t)+\gamma_k^-(t))+\delta(t)
  \end{align*}
  where $\delta:\real\to\real$ is a $\C^{\infty}(\real)$ function such
  that $\delta\equiv0$ on $[-\frac{1}{3},\frac{1}{3}]$, $\delta\geq1$
  on $[\frac{1}{2},\infty)$ and $\delta\leq-1$ on
  $(-\infty,-\frac{1}{2}]$.  By following an argument analogous to the
  one in the proof of~\cite[Lemma 4.3]{YL-EDS-YW:96}, we can show that
  $\gamma$ is smooth in $\real\backslash\{0\}$ and
  $\lim\limits_{t\to0} \gamma^{(i)}(t)=0$ for all $i\geq1$. Moreover,
  $\beta(t):=\int_{0}^{t}\gamma(s)ds$ is an extended class
  $\Kc_{\infty}$, smooth in $\real\backslash\{0\}$, and satisfies
  $\lim\limits_{t\to0} \beta^{(i)}(t)=0$ for all $i\geq1$.
  Again by following the proof of~\cite[Lemma 4.3]{YL-EDS-YW:96}, one
  can show that $\Phi=\beta \circ \Psi$ is smooth in $\real^n$. This
  follows by considering sequences converging to points in
  $\partial\Cc$ and showing that the derivatives of any order of
  $\Phi$ along these sequences converge to zero.

  Finally, now that we know that $\Phi$ is smooth and satisfies
  $\Int(\Cc)=\setdef{x\in\real^n}{\Phi(x)>0}$,
  $\partial\Cc=\setdef{x\in\real^n}{\Phi(x)=0}$, let us show that
  $\Phi$ is a CBF. Indeed, for all $x\in\text{Int}(\Cc)$, 
  since $\Phi=\beta \circ \Psi$, it holds that
  \begin{align*}
    \nabla \Phi(x)^T f(x,u_0(x)) \geq-\beta'(\Psi(x))\bar{\alpha}(\Psi(x))  
  \end{align*}
  Note that $\beta'(r)\bar{\alpha}(r)>0$ for
  $r>0$, $\beta'(0)\bar{\alpha}(0)=0$ and $\beta'(r)\bar{\alpha}(r)$
  is smooth for all $r\in\real$. Therefore, $\beta'(r)\bar{\alpha}(r)$
  can be upper bounded for $r\geq0$ by a smooth extended class
  $\Kc_{\infty}$ function~$\hat{\alpha}$. Finally, since $\beta$ is an
  extended class $\Kc_{\infty}$ function, it is invertible and we can
  define $\breve{\alpha}(r):=\hat{\alpha}\circ\beta^{-1}(r)$, which is
  also smooth and extended class $\Kc_\infty$, so that for all $x\in\Cc$
  it holds that 
  \begin{align*}
    \nabla \Phi(x)^T f(x,u_0(x)) \geq-\breve{\alpha}(\Phi(x)).
  \end{align*}
  Hence $\Phi$ is a CBF of $\Cc$.
  In fact $u_0(x)$ satisfies the CBF inequality~\eqref{eq:cbf-ineq}
  for all $x\in\Cc$.

  \textbf{Let us now show~\ref{it:second}.} First, let
    us show that without loss of generality, we can assume that $h$ is
    bounded in $\Cc$. Indeed, if $h$ is not bounded, consider the
    function $\tilde{h}:\real^n\to\real$ as
    $\tilde{h}(x)=h(x)e^{-h(x)}$. Note that $\tilde{h}$ is
    continuously differentiable (because $h$ is also continuously
    differentiable), bounded in $\Cc$ and, since $e^{h(x)} > 0$ for
    all $x\in\Cc$, $\Cc=\setdef{x\in\real^n}{\tilde{h}(x) \ge 0}$,
    $\partial\Cc=\setdef{x\in\real^n}{h(x)=0}$,
    $\text{Int}(\Cc)=\setdef{x\in\real^n}{h(x)>0}$.  Furthermore,
  \begin{align}\label{eq:gradient-tildeh}
    \nabla\tilde{h}(x) = \nabla h(x) e^{-h(x)} (1-h(x)),
  \end{align}
  for all $x\in\Cc$, and it satisfies an inequality analogous
  to~\eqref{eq:grad-boundedness-h}.  Indeed, let $\tilde{M}>0$ be such
  that $e^{-2h(x)} (1+h(x))^2  < \tilde{M}$ for all
  $x\in\Cc$. Then, by defining $\tilde{b}_k := \tilde{M}b_k$ for
  $k\in\{ 0, 1, \hdots, N_2 \}$, and using~\eqref{eq:gradient-tildeh}
  we have
  \begin{align*}
    \|\nabla\tilde{h}(x)\|^2 \leq \sum_{k=0}^{N_2} \tilde{b}_k
    \norm{ f(x,\hat{u}(x)) }^2, 
  \end{align*}
  for all $x\in\Cc$.  Therefore, without loss of generality, we assume
  that $h$ is bounded in $\Cc$. Let $M>0$ be such that $h(x) \leq M$
  for all $x\in\Cc$. Now, consider the function
$\bar{h}(x)=e^{-\norm{f(x,\hat{u}(x))}^2}h(x)$. Note that $\bar{h}$ is
bounded, continuously differentiable and
$\Cc=\setdef{x\in\real^n}{\bar{h}(x)\geq0}$,
$\Int(\Cc)=\setdef{x\in\real^n}{\bar{h}(x)>0}$ and
$\partial\Cc=\setdef{x\in\real^n}{\bar{h}(x)=0}$. Moreover,
  \small
  \begin{gather*}
    \nabla \bar{h}(x)^T
    f(x,\hat{u}(x))=e^{-\norm{f(x,\hat{u}(x))}^2}\nabla h(x)^T
    f(x,\hat{u}(x))
    \\ 
    -h(x)e^{-\norm{f(x,\hat{u}(x))}^2}
    \nabla(\norm{f(x,\hat{u}(x))}^2)^T f(x,\hat{u}(x)), 
  \end{gather*}
  \normalsize
  and, by using the bounds in~\eqref{eq:grad-boundedness}, the fact
  that $h$ is bounded in $\Cc$, and the Cauchy-Schwartz inequality,
  we get that for all $x\in\Cc$, the following holds:
  \begin{align*}
    & \nabla \bar{h}(x)^T f(x,\hat{u}(x))
      \geq
    \\
    & \quad -e^{-\norm{f(x,\hat{u}(x))}^2} 
      \sqrt{ \sum_{k=0}^{N_2}
      b_k \norm{f(x,\hat{u}(x))}^{k+1} }
    \\
    &\qquad -e^{-\norm{f(x,\hat{u}(x))}^2} M \sqrt{\sum_{j=0}^{M_2} c_j
      \norm{f(x,\hat{u}(x))}^{j+1}}.
  \end{align*}
  Thus,
  \begin{align*}
    \alpha(r)=-\inf_{
    \begin{subarray}{c}
      x \text{ s.t.}
      \\
      \bar{h}(x)\in [0,r]
    \end{subarray}
    } \nabla \bar{h}(x)^T f(x,\hat{u}(x))
  \end{align*}
  is finite for all $r\geq0$. This follows from the fact that even if
  $\norm{f(x,\hat{u}(x))}$ is unbounded as $\norm{x}\to\infty$, the
  exponential term in $\norm{f(x,\hat{u}(x))}$ will dominate the polynomial terms
  in $\norm{f(x,\hat{u}(x))}$.  Moreover,
  since $\hat{u}$ is safe, and $\bar{h}$ satisfies conditions~\eqref{eq:safe-set},
  any trajectory 
  $x(\cdot)$ of $\dot{x}=f(x,\hat{u}(x))$
  for which $x(t^*)\in\partial\Cc$ for some $t^*\in\real$ 
  necessarily satisfies $\nabla\bar{h}(x)^T f(x,\hat{u}(x)) \geq 0$
  by Nagumo's Theorem~\cite{MN:42}.
  This implies that $\alpha(0)\leq0$, and $\alpha$
  can be upper bounded by a class $\Kc_{\infty}$ function
  $\alpha_0$, from where we have
  $\nabla \bar{h}(x)^T f(x,\hat{u}(x)) \geq-\alpha_0(\bar{h}(x))$ for
  all $x\in\Cc$, and hence $\bar{h}$ is a CBF of~$\Cc$.
  In fact $\hat{u}(x)$ satisfies the CBF inequality~\eqref{eq:cbf-ineq}
  for all $x\in\Cc$.
\end{proof}

\begin{remark}\longthmtitle{Minimal CBFs and smoothness
    properties}\label{rem:minimal-cbfs-regularity}
  {\rm
  Even though $0$ is not a regular value of the CBF $\Phi$ constructed
  in Theorem~\ref{thm:converse-cbf}~\ref{it:zero}, $\Phi$ is a minimal
  CBF because the extended class $\Kc_{\infty}$ function
  $\breve{\alpha}$ is smooth~\cite[Corollary 1]{RK-ADA-SC:21}.
      %
      %
  Moreover, if in Theorem~\ref{thm:converse-cbf}~\ref{it:second}, we
  add the assumption that $0$ is a regular value of $h$, then $0$ is
  also a regular value of $\bar{h}$ (the CBF constructed in the proof) 
  and $\bar{h}$ is a minimal 
  CBF~\cite[Section III]{RK-ADA-SC:21}. 
  This implies that the
  CBFs constructed in Theorem~\ref{thm:converse-cbf} can be used for
  control design according to~\cite[Theorem~4]{RK-ADA-SC:21}
  and~\cite[Theorem~2]{ADA-SC-ME-GN-KS-PT:19}.
  Moreover, even if the minimal CBFs constructed in 
  Theorem~\ref{thm:converse-cbf}~\ref{it:second}
  and~\ref{it:third} are not $\Cc^{\infty}(\real^n)$, by applying the smoothing
  procedure outlined in the proof of
  Theorem~\ref{thm:converse-cbf}~\ref{it:zero} to any differentiable
  minimal CBF, we can construct another minimal CBF that is
  $\Cc^{\infty}(\real^n)$.
  \demo
  }
\end{remark}

\begin{remark}\longthmtitle{Class of systems and safe sets
    satisfying~\eqref{eq:grad-boundedness}} {\rm As
      shown in the proof of~\ref{it:second},
      condition~\eqref{eq:grad-boundedness} guarantees that $\bar{h}$
      (as defined therein) satisfies that the function
      $x\to\nabla\bar{h}(x)^T f(x,\hat{u}(x))$ is lower bounded in the
      set $\setdef{x\in\real^n}{\bar{h}(x)\in[0,r]}$, avoiding the issues
      faced in Example~\ref{ex:wrong-candidate}.
    %
    %
    Condition~\eqref{eq:grad-boundedness} is satisfied by a large
    class of systems, including polynomial systems for which there
    exists a polynomial safe feedback and safe sets $\Cc$ for which
    there exists a polynomial function $h$
    satisfying~\eqref{eq:safe-set}.  \demo }
\end{remark}

\begin{remark}\longthmtitle{Comparison with time-varying barrier
    functions}\label{rem:comparison-time-varying-bfs} 
  {\rm The result in~\cite[Theorem 2]{MM-RGS:23} guarantees the existence
    of a time-varying barrier function, cf.~\cite[Definition
    15]{MM-RGS:23}, under more general assumptions than
    Theorem~\ref{thm:converse-cbf}.  Despite the importance of this
    result, the construction of such time-varying barrier functions is
    in general complicated and requires computing an appropriately
    defined reachable set.  Moreover, such functions are in general
    not differentiable even if the dynamics are smooth~\cite[Theorem
    4]{MM-RGS:23}.  The added time dependence and the lack of control
    input and extended class $\Kc$ function make the notion of
    time-varying barrier functions substantially different from the
    notion of control barrier function considered here, which is the
    one widely employed in the safety-critical control
    literature~\cite{ADA-SC-ME-GN-KS-PT:19,ADA-XX-JWG-PT:17,SCH-XX-ADA:15}.
    It is also not apparent from the proof of~\cite[Theorem
    2]{MM-RGS:23} when the obtained barrier function is
    time-independent.  We also point out that the proof technique in
    Theorem~\ref{thm:converse-cbf} is different from that
    of~\cite[Theorem 2]{MM-RGS:23}, and as pointed out in
    Remark~\ref{rem:minimal-cbfs-regularity}, can be used to obtain
    barrier functions which are $\Cc^{\infty}(\real^n)$.}
  \demo
      %
  %
\end{remark}

  \begin{remark}\longthmtitle{Robust safety}\label{rem:comparison-robust-safety}
    {\rm Here we comment on the relationship between robust safety,
      cf.~\cite[Definition 2]{MM-MG:22}, and the assumptions in
      Theorem~\ref{thm:converse-cbf}~\ref{it:zero}, motivated by the
      fact that robust safety is also a sufficient condition for the
      existence of a certain notion of barrier function (different
      from the one adopted here because it does not utilize extended
      class $\Kc_{\infty}$ functions)~\cite[Theorems~1
      and~2]{MM-MG:22}.
      %
      %
      In particular, next we provide an example where the conditions
      in Theorem~\ref{thm:converse-cbf}~\ref{it:second} hold but
      robust safety does not hold.  Consider the scalar system
      $\dot{x} = xu$, safe set $\Cc = \setdef{x\in\real}{x \geq 0}$,
      and safe controller $k(x) = 1$ for all $x\in\real$.  Any
      trajectory of the closed-loop system with initial condition in
      $\text{Int}(\Cc)$ diverges to infinity and therefore the
      assumptions of Theorem~\ref{thm:converse-cbf}~\ref{it:second}
      hold.  However, the system is not robustly safe, since for any
      scalar function $\epsilon:\real\to\real_{>0}$, the trajectory of
      $\dot{x} = x - \epsilon(x)$ with initial condition at the origin
      enters the unsafe set (because $\epsilon(0)>0$).  \demo }
  \end{remark}
%
%

\begin{remark}\longthmtitle{Strict positivity in the interior of the
    safe set}
  {\rm 
  Definition~\ref{def:cbf} requires $h$ to be strictly positive in
  $\Int(\Cc)$. However, other definitions available in the literature
  (e.g.,~\cite[Theorem 1]{RK-ADA-SC:21}) only require
  $\Cc:=\setdef{x\in\real^n}{h(x)\geq0}$. With this definition of CBF
  it can be shown that any safe set admits a CBF. This follows by
  adapting the proof of Theorem~\ref{thm:converse-cbf}~\ref{it:zero}
  to the case where there might be points $x\in\Int(\Cc)$ for which
  $V_{-\infty}(x)=0$, i.e., finding a smooth approximation $\Psi$ of
  $V_{-\infty}(x)$ at $\setdef{x\in\real^n}{V_{-\infty}(x)\neq0}$ and
  then extending $\Psi$ smoothly at
  $\setdef{x\in\real^n}{V_{-\infty}(x)=0}$ as done in the proof of
  Theorem~\ref{thm:converse-cbf}~\ref{it:zero}.  \demo}
\end{remark}
%
%
\begin{remark}\longthmtitle{Asymptotic stability of safe
    set}\label{rem:asymptotic-stab-safe-set} {\rm As shown
    in~\cite[Proposition 2]{ADA-XX-JWG-PT:17}, if
    $u_{\text{sf}}:\real^n\to\real^m$ is a Lipschitz controller satisfying the CBF
    condition~\eqref{eq:cbf-ineq} for all $x\in\Dc$, where $\Dc$ is an
    open set containing $\Cc$, then the set $\Cc$ is asymptotically
    stable. By appropriately strengthening the conditions in
    Theorem~\ref{thm:converse-cbf}, we can also guarantee the
    existence of a CBF valid in an open set containing $\Cc$, and
    hence certifying asymptotic stability of $\Cc$.
    Indeed,
    \begin{enumerate}
    \item in
      Theorem~\ref{thm:converse-cbf}~\ref{it:minus1}, if there exists
      $r_0<0$ such that~\eqref{eq:cbf-condition-with-inf} is satisfied
      for all $r\geq r_0$, then $h$ is a CBF and~\eqref{eq:cbf-ineq}
      is feasible (by using $u=u_*(x)$) for all $x\in\real^n$ with
      $h(x)\geq r_0$;
    \item in Theorem~\ref{thm:converse-cbf}~\ref{it:zero}, under the 
      additional assumption that there exists an open set  
    $\Dc$ containing $\Cc$ for which any trajectory of
    $\dot{x}=f(x_0,u_0(x))$ with initial condition  
    in $\Dc\backslash\text{Int}(\Cc)$ either converges to $\Cc$ in
    finite time or asymptotically, 
    by following the same proof technique we can show that $\Phi$ (as
    obtained in the proof) 
    is a CBF and~\eqref{eq:cbf-ineq} is feasible (by using $u=u_0(x)$)
    for all $x\in\Dc$;    
  \item in Theorem~\ref{thm:converse-cbf}~\ref{it:second}, if there
    exists an open set $\Dc$ containing $\Cc$ for
    which~\eqref{eq:grad-boundedness}  
    holds for all $x\in\Dc$, then the same proof technique shows that
    there exists a CBF of $\Cc$ and the  
    corresponding inequality~\eqref{eq:cbf-ineq} is also feasible in
    $\Dc$;
    \item in Theorem~\ref{thm:converse-cbf}~\ref{it:third}, since
      $\Cc$ is compact,  
    if it is asymptotically stable, by~\cite[Theorem 2.9]{YL-EDS-YW:96}
    there exists a Lyapunov function with respect to $\Cc$. Therefore, 
    $h$ (which is guaranteed to be a CBF by
    Theorem~\ref{thm:compact-converse-cbf}) 
    can be extended outside of $\Cc$ using this Lyapunov function
    (smoothly, using the arguments 
    in~\cite[Proposition 4.2]{YL-EDS-YW:96} and
    Theorem~\ref{thm:converse-cbf}~\ref{it:zero}).   \demo
  \end{enumerate}
}
\end{remark}

Even though Theorem~\ref{thm:converse-cbf} extends significantly
Theorem~\ref{thm:compact-converse-cbf} regarding the class of safe
sets and systems for which a CBF exists, it is an open problem to
determine whether an even more general result holds true. 
{\color{blue} The following example is not covered by any of the cases in
Theorem~\ref{thm:converse-cbf}.}

  \begin{example}\longthmtitle{Example not covered by converse CBF
      theorem}\label{ex:example-not-covered-by-converse-CBF-theorem}
    {\rm Here we provide an example of a control system and safe set
      not covered by any of the cases in
      Theorem~\ref{thm:converse-cbf}. Let $h:\real\to\real$ be defined
      as $h(x) = e^x \sin(x)$, and let $\Cc$ be the corresponding safe
      set as defined in~\eqref{eq:safe-set}.  Note that $h$ is
      continuously differentiable.  Define the dynamics by letting
      $f(x,u) = h(x)$ (since systems without control are a special
      case of systems with control, Theorem~\ref{thm:converse-cbf}
      still applies).  Note that $\Cc$ is safe because all points in
      $\partial\Cc$ are equilibrium points, and since $h$ is
      continuously differentiable, trajectories of the dynamical
      system are unique, which means that no trajectory can leave the
      safe set.  Furthermore, since $\dot{x} > 0$ whenever
      $x\in\text{Int}(\Cc)$, trajectories of the dynamical system with
      initial condition in $\text{Int}(\Cc)$ converge to $\partial\Cc$
      and therefore item ~\ref{it:zero} does not hold.  Similarly,
      since $\Cc$ is not compact, item~\ref{it:third} does not hold
      either.  Now let us show that item~\ref{it:minus1} does not
      hold, which in turn also implies that item~\ref{it:second} can
      not hold (because item~\ref{it:second} implies
      item~\ref{it:minus1}, as shown in the proof of Theorem IV.3).
      For any $k\in\mathbb{Z}_{>0}$, let
      $\underline{a}_k = (2k+1)\pi-\frac{\pi}{4}$,
      $\bar{a}_k = (2k+1)\pi$, and note that
  \begin{align*}
    h( \underline{a}_k ) \!\geq\! \frac{e^{-\frac{\pi}{4}} \sqrt{2}}{2}, \
    \cos( \underline{a}_k ) \!=\! -\frac{\sqrt{2}}{2}, \
    h( \bar{a}_k )\!=\! 0, \ \cos( \bar{a}_k ) \!=\! -1,
  \end{align*}
  and $\cos(a) < 0$ for all $a\in[\underline{a}_k, \bar{a}_k]$.
  Therefore, by letting $r = e^{-\frac{\pi}{4}}\frac{\sqrt{2}}{2}$, there exists 
  a sequence $(a_k)_{k\in\mathbb{Z}_{>0}}$
  (with $a_k \in ( (2k+1)\pi-\frac{\pi}{4}, (2k+1)\pi )$ for all $k\in\mathbb{Z}_{>0}$)
  such that 
  $\frac{r}{2} < h(a_k) < r$,
  and $\cos(a_k) < 0$ for all $k\in\mathbb{Z}_{>0}$.
  Now, note that 
  \begin{align*}
    h^{\prime}(a_k) f(a_k) &= h(a_k)^2 + e^{a_k}\cos(a_k)h(a_k).
  \end{align*}
  Since $\sin(a_k) < \frac{r}{e^{a_k}} < 1$, it follows that 
  $|\cos(a_k)| = \sqrt{1-\sin^2(a_k)} > \sqrt{1-\Big( \frac{r}{e^{a_k}} \Big)^2 }$,
  and 
  \begin{align*}
    h^{\prime}(a_k) f(a_k) < r^2 - e^{a_k} \frac{r}{2} \sqrt{1-\Big( \frac{r}{e^{a_k}} \Big)^2 }.
  \end{align*}
  Therefore, since $\lim\limits_{k\to\infty} a_k = \infty$,
  \begin{align*}
    \lim\limits_{k\to\infty} h^{\prime}(a_k) f(a_k) = -\infty,
  \end{align*}
  which means that~\ref{it:minus1} does not hold.  }
  \demo
\end{example}


\subsection{Extended Control Barrier
  Functions}\label{sec:extended-cbf}

In this section, we show that Theorem~\ref{thm:compact-converse-cbf}
remains valid when we drop the compactness assumption provided that
one employs a slight generalization of the notion of CBF. The latter
requires a generalization of the notion of extended
class~$\Kc_{\infty}$.

\begin{definition}\longthmtitle{Extended class $\Kc \Kc$ function}
  A continuous function
  $\alpha:\real \times \real_{\geq0} \to \real_{\geq0}$ is of class
  $\Kc \Kc$ if $\alpha(\cdot,s)$ and $\alpha(r,\cdot)$ are strictly
  increasing for all $s\geq0$, $r\in\real$, respectively and
  $\alpha(0,s)=0$ for all $s\geq0$. It is of extended class
  $\Kc_{\infty}\Kc$ if, additionally,
  $\lim_{r\to\pm\infty}\alpha(r,s)=\pm\infty$ for all $s\geq0$.
\end{definition}

We are ready to introduce the notion of \textit{Extended Control
  Barrier Functions}.

\begin{definition}\longthmtitle{Extended Control Barrier Function}
  Let $h:\real^{n}\to\real$ be a continuously differentiable function
  and let $\Cc$ be defined as in~\eqref{eq:safe-set}. The function $h$
  is an \textit{extended control barrier function} (\textbf{eCBF}) of
  $\Cc$ if there exists an extended class
  $\mathcal{K}_{\infty}\mathcal{K}$ function $\alpha$ such that, for
  all $x\in\Cc$, there exists a control $u\in\real^{m}$ satisfying:
  \begin{align}\label{eq:ecbf-condition}
    \nabla h(x)^T f(x,u) \geq -\alpha(h(x),\norm{x}).
  \end{align}
\end{definition}
\smallskip

Note that eCBFs allow for the time derivative of $h$ to become
arbitrarily negative as $\norm{x}$ approaches infinity, as long as
such derivative stays nonnegative at the boundary of~$\Cc$. 
The following result relates the notions of CBF and eCBF.

\begin{proposition}\longthmtitle{Relationship between CBFs and
    eCBFs}\label{prop:cbfs-ecbfs}
  A CBF of~$\Cc$ is also an eCBF of~$\Cc$.  Moreover, if $\Cc$ is
  compact, an eCBF of~$\Cc$ is a CBF of~$\Cc$.
\end{proposition}
\begin{proof}
  Let $h$ be a CBF of $\Cc$, i.e., there exists an extended class
  $\Kc_{\infty}$ function $\alpha$ such that, for all $x\in\Cc$, there
  exists $u\in\real^m$ satisfying~\eqref{eq:cbf-ineq}.  Define 
  \begin{align*}
    \hat{\alpha}(r,s):=\alpha(r)(s+1)
  \end{align*}
  and note that it is an extended class $\Kc_{\infty}\Kc$ function.
  Since $\hat{\alpha}(h(x),\norm{x})\geq\alpha(h(x))$ for all
  $x\in\Cc$, it follows that for all $x\in\Cc$ there exists
  $u\in\real^m$ satisfying~\eqref{eq:ecbf-condition}.

  Now, suppose that $\Cc$ is compact and let $h_e$ be an eCBF
  of~$\Cc$, i.e., there exists a class $\Kc_{\infty}\Kc$ function
  $\alpha_e$ such that, for all $x\in\Cc$, there exists $u\in\real^m$
  satisfying~\eqref{eq:ecbf-condition}.
  Define, for $r\geq0$,
  \begin{align*}
    \tilde{\alpha}_e(r) := \sup\limits_{ \setdef{x\in\real^n}{\ 0\leq
    h_e(x)\leq r} } \alpha_e(r,\norm{x}). 
  \end{align*}
  Since $\Cc$ is compact, $\tilde{\alpha}_e(r)$ is finite for all
  $r\geq0$. Furthermore, it is strictly increasing, satisfies
  $\tilde{\alpha}_e(0) = 0$, and satisfies
  $\lim\limits_{r\to\infty}\alpha(r)=\infty$, so it can be extended 
  to $\real_{<0}$ so that is of
  extended class $\Kc_{\infty}$. Since for all $x\in\Cc$,
  $\tilde{\alpha}_e(h(x))\geq \alpha(h(x),\norm{x})$, we have that, for
  all $x\in\Cc$, there exists $u\in\real^m$
  satisfying~\eqref{eq:cbf-ineq}, i.e.,
  $h_e$ is a CBF.
\end{proof}

According to Proposition~\ref{prop:cbfs-ecbfs}, eCBFs coincide with
CBFs when the safe set is compact.  As we show later, the notion of
eCBF is a more suitable notion to deal with safe sets that are
unbounded.  Similarly to Definition~\ref{def:compat-clf-cbf}, we can
also define a notion of compatibility for eCBFs (instead of CBFs) and CLFs.

\begin{definition}
  A CLF $V$ and an eCBF $h$ are \textbf{compatible} at $x\in\Cc$ if
  there exists $u\in\real^{m}$ satisfying~\eqref{eq:clf-ineq}
  and~\eqref{eq:ecbf-condition} simultaneously. We refer to both
  functions as compatible in $\Cc$ if they are compatible at every
  point in $\Cc$.
\end{definition}

Since eCBFs also enforce the satisfaction of Nagumo's
Theorem~\cite{MN:42}, they can be used to certify safety, as stated in
the following result. We omit its proof, which follows an argument
analogous to that of~\cite[Theorem 2]{ADA-SC-ME-GN-KS-PT:19}.


\begin{proposition}\longthmtitle{eCBFs certify safety}
  Let $\Cc\subset\real^n$, $h$ an eCBF of $\Cc$, and $0$ a regular
  value of $h$.  Any Lipschitz controller $k:\real^n\to\real^m$ that
  satisfies
  $k(x)\in K_{\text{ecbf}}(x):=\setdef{u\in\real^m}{\nabla h(x)^T
    f(x,u)+\alpha(h(x),\norm{x})\geq0}$ for all $x\in\Cc$ renders the
  set $\Cc$ forward invariant.
\end{proposition}

We also point out that the control designs proposed
in~\cite{PO-JC:19-cdc,PM-JC:23-csl,ADA-XX-JWG-PT:17} can easily be
adapted using eCBFs instead of CBFs.



Next we show that the flexibility added by the class $\Kc_{\infty}\Kc$
function allows eCBFs to resolve some of the issues faced by CBFs.

\begin{example}[Examples~\ref{ex:wrong-candidate} and~\ref{ex:example-not-covered-by-converse-CBF-theorem}
  revisited]\label{ex:wrong-candidate-ecbf}
  {\rm
  We show here that $h(x,y)=x$ is an eCBF for
  Example~\ref{ex:wrong-candidate}. Take $\alpha(r,s)=rs$ as the
  extended class $\Kc_\infty\Kc$ function
  in~\eqref{eq:ecbf-condition}. It is straightforward to check that
  $\nabla h(x,y)^{T} \big(\begin{smallmatrix} xy+1 \\
    -y+u \end{smallmatrix}\big) = xy+1\geq
  -x\sqrt{x^2+y^2}=-\alpha(h(x,y),\norm{(x,y)})$ for $x\geq0$ and
  hence~\eqref{eq:ecbf-condition} is satisfied for all points in
  $\Cc$.  
  By a similar argument, the function $h$ defined in~\ref{ex:example-not-covered-by-converse-CBF-theorem}
  is also an eCBF for the dynamics defined therein.
  \problemfinal
  }
\end{example}

The following result states that the existence of an eCBF is also
necessary for a set to be safe, generalizing
Theorem~\ref{thm:compact-converse-cbf} to safe sets that might be unbounded.

\begin{theorem}\longthmtitle{Converse eCBF
    result for safe
    sets}\label{thm:converse-ecbf}
  Given a control system~\eqref{eq:control-sys}, let
  $h:\real^{n}\to\real$ be a continuously differentiable function with
  $\Cc$ defined as in~\eqref{eq:safe-set} and with $0$ a regular value
  of $h$. If $\Cc$ is safe, then $h$ is an eCBF.
\end{theorem}
\begin{proof}
  For $r\geq0$, $c\geq0$, define the set
  $S_{r,c}:=\setdef{x\in\real^n}{0\leq h(x)\leq r, \norm{x}\leq
    c+c_{\min}}$, where $c_{\min}\geq0$ is taken so that
  $S_{0,0}\neq\emptyset$ (for instance, one can set $c_{\min}$ as the
  distance from the origin to $\partial\Cc$). Since
  $S_{0,0}\subseteq S_{r,c}$ for any $r\geq0$, $c\geq0$, this
  guarantees that $S_{r,c}\neq\emptyset$ for all $r\geq0,
  c\geq0$. Next, define
  \begin{align}
    \hat{\alpha}(r,c):=-\min_{x\in S_{r,c}} \nabla h(x)^T
    f(x,\hat{u}(x)) ,
  \end{align}
  where $\hat{u}:\real^{n}\to\real^{m}$ is a locally Lipschitz
  controller that renders $\Cc$ safe (which exists, by
  assumption). Since $S_{r,c}$ is compact for all $r\geq0$, $c\geq0$,
  $\hat{\alpha}(r,c)$ is finite for all $r\geq0$, $c\geq0$.  Note that
  $\hat{\alpha}$ is non-decreasing in both its first and second
  arguments. Moreover, since $\Cc$ is forward invariant under
  $\dot{x}=f(x,\hat{u}(x))$ and $0$ is a regular value of $h$, by
  Nagumo's theorem~\cite{MN:42}, $\hat{\alpha}(0,c)\leq0$ for all
  $c\geq0$.
  Hence, we can find a class $\mathcal{K}_{\infty}\mathcal{K}$
  function $\alpha$ such that $\alpha(r,c)\geq\hat{\alpha}(r,c)$ for
  all $r\geq0,c\geq0$. This ensures that
  \begin{align*}
    \nabla h(x)^T f(x,\hat{u}(x)) \geq -\alpha(h(x),\norm{x}).
  \end{align*}
  for all $x\in\Cc$, hence completing the proof.
\end{proof}

\begin{remark}\longthmtitle{Comparison with time-varying barrier
    functions-cont'd}\label{rem:comparison-time-varying-bfs-continued}
  {\rm
  As mentioned in
  Remark~\ref{rem:comparison-time-varying-bfs},~\cite[Theorem~2]{MM-RGS:23}
  provides a necessary and sufficient condition for safety in terms of
  so-called time-varying barrier functions, which might however be
  difficult to construct and utilize in practice to design safe
  controllers.  Instead, in the less general setting considered here,
  Theorem~\ref{thm:converse-ecbf} ensures that if $\Cc$ is safe, any
  scalar continuously differentiable function
  satisfying~\eqref{eq:safe-set} and having $0$ as a regular value is
  an eCBF.  This ensures that $h$ is time-invariant and continuously
  differentiable, and instead of computing a complicated reachable
  set, only requires finding a scalar continuously differentiable
  function satisfying~\eqref{eq:safe-set}, with $0$ being a regular
  value of~it.  \demo
  }
\end{remark}


\section{Converse Theorems for Joint Safety and
  Stability}\label{sec:converse-sf-st}

In this section we address problems (P2) and (P3) in
Section~\ref{sec:problem-statement}.  We start by studying under what
conditions the existence of either (i) a compatible CLF-CBF pair or
(ii) a CLBF is guaranteed.  Our motivation comes from the fact that in
either case locally Lipschitz feedback controllers that
achieve safe stabilization can be designed under appropriate technical
conditions, cf.  Section~\ref{sec:preliminaries}.
Throughout this section, we assume $\mathbf{0}_n\in\text{Int}(\Cc)$.
%
%


Our starting point is the result in~\cite[Theorem 11]{PB-CMK:17},
which shows that if the set $\real^{n}\backslash\Cc$ is
bounded, then a CLBF of $\real^n\backslash\Cc$ can not exist. 
In fact the proof of~\cite[Theorem 11]{PB-CMK:17} shows that if 
$\real^n\backslash\Cc$ is bounded, a
locally Lipschitz safe stabilizing controller can not exist.  More
generally, the same argument shows that if $\real^{n}\backslash\Cc$
has a bounded connected component, then a safe stabilizing controller
can not exist.

\begin{proposition}\longthmtitle{Topological obstruction for existence
    of compatible CLF-CBF pair}\label{prop:topological-obstruction}
  If the set $\real^n\backslash\Cc$ has a bounded connected component,
  then there does not exist a strictly compatible CLF-CBF pair
  on~$\Cc$.
\end{proposition}
\begin{proof}
  Suppose there exits a strictly compatible CLF-CBF pair on
  $\Cc$. Then, by using the universal formula in~\cite{PO-JC:19-cdc}
  (which is constructed by computing the centroid of the set of controls satisfying 
  the CLF and CBF conditions),
  one can construct a smooth safe stabilizing controller on
  $\Cc$. 
  But, since $\real^n\backslash\Cc$ contains a bounded
  connected component, by the argument used in~\cite[Theorem
  11]{PB-CMK:17}, this can not be possible, reaching a contradiction.
\end{proof}

Even though Proposition~\ref{prop:topological-obstruction} only
ensures the non-existence of a \textit{strictly} compatible CLF-CBF
pair, it also shows that even if a compatible CLF-CBF pair exists, one
would not be able to leverage it to design a locally Lipschitz
controller that safely stabilizes the system. The above result
explains why the recent body of
literature~\cite{MFR-APA-PT:21,XT-DVD:24,PM-JC:23-csl,YC-PM-EDA-JC:24-cdc,PM-YC-EDA-JC:25-jnls}
on locally Lipschitz controllers that achieve safe stabilization
obtain closed-loop systems with undesirable equilibrium points in the
boundary of the safe set, when the set of unsafe states has a bounded
connected component. We note also that the proof of
  Proposition~\ref{prop:topological-obstruction} relies
  on~\cite[Theorem 11]{PB-CMK:17}, which shows that if
  $\real^n\backslash\Cc$ is bounded, there can not exist a smooth safe
  stabilizing controller. A similar result (for analytic vector
  fields) is also available in~\cite[Proposition 3]{DEK:87-icra}.

The next result identifies another scenario where a CLBF does not
exist.

\begin{proposition}\longthmtitle{No CLBF exists for unbounded safe
    sets}\label{prop:no-clbf-unbounded}
  Suppose $\Cc\neq\real^{n}$ is unbounded. Then, there does not exist
  a CLBF of $\real^{n}\backslash\Cc$.
\end{proposition}
\begin{proof}
  If $\real^n\backslash\Cc$ is bounded, there does not exist a CLBF of
  $\real^{n}\backslash\Cc$ by~\cite{PB-CMK:20}. If
  $\real^n\backslash\Cc$ is unbounded, assume by contradiction that
  there exists a CLBF $\bar{V}$. As shown in~\cite[Remark
  13]{YM-YL-MF-JL:22}, condition~\eqref{eq:empty-intersection}
  requires that $\partial\Cc$ is the $0$-level set of $\bar{V}$.
  Indeed, if $x\in\partial\Cc$ is such that
    $\bar{V}(x)>0$, then there exists a sequence
    $\{ x_n \}_{n\in\mathbb{Z}_{>0}}$ converging to $x$ such that $\bar{V}(x_n)>0$ (and
    hence $x_n\notin\Uc$) and $x_n\in\Cc$ for all
    $n\in\mathbb{Z}_{>0}$.  This means that
    $x\in\overline{(\Cc\backslash\Uc)}\cap\overline{(\real^n\backslash\Cc)}$,
    which is impossible by condition~\eqref{eq:empty-intersection}.
  Finally, note that it is not possible for $\partial\Cc$ to be a
  $0$-level set of $\bar{V}$ because $\bar{V}$ is proper, implying that all
  of its level sets are compact.
\end{proof}
%
%

Next, we turn our attention to identifying conditions under which either a
CLBF or a compatible CLF-CBF pair exists provided that the origin is
safely stabilizable.

\begin{theorem}\longthmtitle{Converse result on safe
    stabilization}\label{thm:converse-safe-stab}
  Given a control system~\eqref{eq:control-sys}, let $\Cc$ be a set
  for which there exists a continuously differentiable function
  $h:\real^n\to\real$ satisfying~\eqref{eq:safe-set}.  Then,
  \begin{enumerate}
  \item\label{it:converse-safe-stab-clbf} if $\Cc$ is compact, $h$ is
    proper, and there exists a locally Lipschitz controller
    $u_{\text{str}}:\real^{n}\to\real^{m}$ such that
    $\nabla h(x)^{T} f(x,u_{\text{str}}(x))>0$ for all
    $x\in\partial\Cc$, as well as a stabilizing controller
    $u_{\text{st}}:\real^n\to\real^m$ such that the origin is
    asymptotically stable for the closed-loop system
    $\dot{x}=f(x,u_{\text{st}}(x))$ with region of
      attraction containing~$\Cc$, then there exists a CLBF of
    $\real^{n}\backslash\Cc$ and a strictly compatible CLF-CBF pair in
    $\Cc$;
  \item\label{it:converse-safe-stab-clf-cbf} if the origin is safely
    stabilizable on $\Cc$ with $u_{\text{ss}}$ a safe stabilizing
    controller, and either the condition in
    Theorem~\ref{thm:converse-cbf}~\ref{it:minus1} holds with 
    $u_* = u_{\text{ss}}$, the condition in
    Theorem~\ref{thm:converse-cbf}~\ref{it:zero} holds with
    $u_0 = u_{\text{ss}}$, the condition in
    Theorem~\ref{thm:converse-cbf}~\ref{it:second} holds with
    $\hat{u}=u_{\text{ss}}$, or the condition in
    Theorem~\ref{thm:converse-cbf}~\ref{it:third} holds, then there
    exists a compatible CLF-CBF pair on $\Cc$;
  \item\label{it:converse-safe-stab-clf-ecbf} if the origin is safely
    stabilizable on $\Cc$, there exists a compatible CLF-eCBF pair on
    $\Cc$.
  \end{enumerate}
\end{theorem}
\begin{proof}
  \textbf{We first show \ref{it:converse-safe-stab-clbf}.}  Note that
  $u_{\text{str}}$ is a safe controller.  Since for all
  $x\in\partial\Cc$, $\nabla h(x)^{T}f(x,u_{\text{str}}(x))>0$, $\Cc$
  is compact, and $\nabla h(x)^T f(x,u_{\text{str}}(x))$ is continuous
  as a function of $x$, there exists $\epsilon>0$ such that
  $\nabla h(x)^{T}f(x,u_{\text{str}}(x))>0$ over
  $\Tc:=\setdef{x\in\real^n}{0\leq h(x)\leq\epsilon}$ and such that
  $\mathbf{0}_n\notin\Tc$ (this is possible because by assumption, 
  $\mathbf{0}_n\in\Int(\Cc)$). Since the
  origin is asymptotically stable for the closed-loop system
  $\dot{x}=f(x,u_{\text{st}}(x))$ with region of attraction containing
  an open set containing $\Cc$, and since the region of attraction 
  is an open set~\cite[Lemma 8.1]{HK:02}, by~\cite[Theorem~4.17]{HK:02}, this
  implies that there exists a CLF $V$ on an open set containing~$\Cc$
  and, furthermore, $\nabla V(x)^T f(x, u_{\text{st}}(x)) < 0$ for all
  $x\in\Cc\backslash\{ \mathbf{0}_n \}$. Since $V$ and $h$ are continuous, $\Cc$
  is compact and
  $\setdef{x\in\real^n}{h(x)=\frac{\epsilon}{2}} \subset \Tc$, there
  exists $\lambda>0$ sufficiently large such that
  $\setdef{x\in\real^n}{\frac{1}{\lambda}V(x)+h(x)=\frac{\epsilon}{2}}
  \cap \Tc \neq \emptyset$.  Let
  $\Pi =
  \setdef{x\in\real^n}{\frac{1}{\lambda}V(x)+h(x)\geq\frac{\epsilon}{2}
  }\cap\Cc$.  Figure~\ref{fig:CLBF-diagram} illustrates the different
  sets defined up to this point.
  %
      %
  Since
  $\setdef{x\in\real^n}{\frac{1}{\lambda}V(x)+h(x)=\frac{\epsilon}{2}}
  \cap \Tc \neq \emptyset$, it follows that $\Pi$ is nonempty, is
  contained in $\Cc$, and is compact.
      %
  %
  Note that $\mathbf{0}_n\in\Pi$ (because $\mathbf{0}_n\in\Cc\backslash\Tc$). Now, define
  \begin{align*}
    \tilde{V}(x)
    =
    \begin{cases}
      \frac{1}{\lambda}V(x)+h(x) \quad &\text{if} \quad  x\in\Pi,
      \\
      \frac{\epsilon}{2} \quad &\text{else}.
    \end{cases}
  \end{align*}
  Recall that $\nabla V(x)^T f(x, u_{\text{st}}(x))<0$ for all
  $x\in\text{Int}(\Pi)\backslash\{ \mathbf{0}_n \}$.  By smoothing $\tilde{V}$
  using the smoothing argument in the proof of
  Theorem~\ref{thm:converse-cbf}~\ref{it:zero}, there exists a smooth
  function $\Psi:\real^{n}\to\real$ such that
  \begin{align*}
    \nabla \Psi(x)^T f(x, u_{\text{st}} (x))
    &\leq \big(
      \frac{1}{\lambda} \nabla V(x) + \nabla h(x) )^T f(x,
      u_{\text{st}} (x) \big)
    \\
    & - \frac{1}{2\lambda} \nabla V(x)^T f(x,u_{\text{st}}(x)),
  \end{align*}
  for all $x\in\text{Int}(\Pi)$, and $\Psi(x) =
  \frac{\epsilon}{2}$, $\nabla
  \Psi(x)=0$ for all
  $x\in\real^{n}\backslash\Int(\Pi)$.  Next, we show that $\bar{V}(x)
  =
  -h(x)+\Psi(x)-\frac{\epsilon}{2}$ is a CLBF of
  $\real^n\backslash\Cc$.  First, note that
  $\bar{V}$ is proper because $h$ is proper and $\bar{V}(x) =
  -h(x)$ for all
  $x\in\real^n\backslash\text{Int}(\Pi)$, $\bar{V}(x)>0$ for all $x
\in\real^{n}\backslash\Cc$, and hence~\eqref{eq:positive}
holds. Moreover, for
$x\in\Cc\backslash\Int(\Pi)$, since
$\nabla\bar{V}\equiv0$ and
$\Cc\backslash\Int(\Pi)\subset\Tc$, it follows that $\nabla
\bar{V}(x)^T f(x,u_{\text{str}}(x)) = -\nabla h(x)^{T}
f(x,u_{\text{str}}(x))<0$.  For $x\in\Int(\Pi)\backslash\{ \mathbf{0}_n \}$,
  \begin{align*}
    &\nabla\bar{V}(x)^T f(x,u_{\text{st}}(x))=(-\nabla h(x) +\nabla\Psi(x))^T
      f(x,u_{\text{st}}(x))
    \\
    &\quad \leq \frac{1}{2\lambda}\nabla V(x)^T f(x,u_{\text{st}} (x))<0.
  \end{align*}
  Hence,~\eqref{eq:clbf-negative-derivative} holds.  Moreover, note
  that $\Cc\backslash\Pi\neq\emptyset$
  and $\bar{V}(x)=-h(x)<0$ in $\Cc\backslash\Pi$. Hence,
  $\Uc:=\setdef{x\in\real^n}{\bar{V}(x)\leq0}\neq\emptyset$
  and~\eqref{eq:nonempty-sublevel-set} holds.  Moreover, since again
  $\bar{V}(x)=-h(x)$ in $\Cc\backslash\Pi$, $\Cc\backslash\Uc\subset\Pi$.
  This means that
  $\overline{(\Cc\backslash\Uc)}\cap\overline{(\real^n\backslash\Cc)}=\emptyset$
  and~\eqref{eq:empty-intersection} holds.  This shows that $\bar{V}$ is
  a CLBF of $\real^{n}\backslash\Cc$.
      %
  \begin{figure}[htb]
    \centering
    \includegraphics[width=0.45\textwidth]{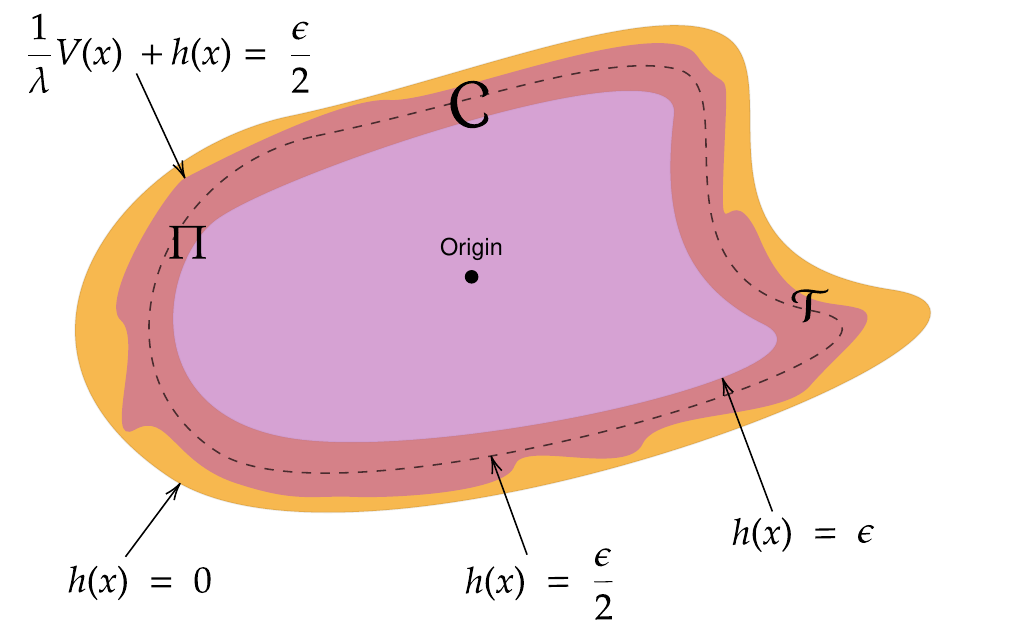}
    \caption{Illustration of the different sets defined in the proof
      of
      Theorem~\ref{thm:converse-safe-stab}\ref{it:converse-safe-stab-clbf}.
      The set
        $\Cc$ is the union of the orange, dark purple and light purple
        regions.  The set
        $\Tc$ is the union of the orange and dark purple regions, and
        the set
        $\Pi$ is the union of the dark and light purple regions.
      }\label{fig:CLBF-diagram}
  \end{figure}
  
  Next, let us show that there exists a strictly compatible CLF-CBF
  pair in $\Cc$.  We do so by using the CLBF
  $\bar{V}$.  Since
  $\bar{V}$ is lower-bounded, it achieves its minimum value at a point
  $p$.  Note that
  $p$ must be the origin, because otherwise,
  by~\eqref{eq:clbf-negative-derivative}
  $\nabla\bar{V}(p)\neq0$, which would mean that
  $p$ is not a local minimum.  Let
  $\hat{V}(x)=\bar{V}(x)-\bar{V}(0)$. Note that
  $\hat{V}$ is proper, positive definite and, for each
  $x\in\Cc$, there exists a control
  $u\in\real^{m}$ satisfying~\eqref{eq:clf-ineq} strictly. Indeed,
  this follows by considering a controller
  $\breve{u}:\real^n\to\real^m$ satisfying $\nabla \bar{V}(x)^T
  f(x,\breve{u}(x))<0$ for all
  $x\in\Cc\backslash\{ \mathbf{0}_n \}$ and taking $W(x):=-\frac{1}{2}\bar{V}(x)^T
  f(x,\breve{u}(x))$ for $x\neq\mathbf{0}_n$ and
  $W(0)=0$ in Definition~\ref{def:clf}. Therefore,
  $\hat{V}$ is a CLF. Now, let
  $\hat{h}(x)=-\bar{V}(x)$. It is easy to check that
  $\hat{h}$ is a CBF of $\Cc$ because it is a candidate CBF of $\Cc$ and 
  $\Cc$ is compact (cf. Theorem~\ref{thm:compact-converse-cbf}).
  Now,~\eqref{eq:clf-ineq} and~\eqref{eq:cbf-ineq} read as
  \begin{subequations}
    \begin{align}
      \nabla \bar{V}(x)^T f(x,u)+W(x)\leq0 ,
      \\
      -\nabla \bar{V}(x)^T f(x,u)+\alpha(-\bar{V}(x))\geq0.
    \end{align}
    \label{eq:clbf-compat}
  \end{subequations}
  Now note that $\breve{u}(x)$ satisfies~\eqref{eq:clbf-compat}
  strictly for all $x\in\Cc$. Hence, $\hat{V}$ and $\hat{h}$ are a
  strictly compatible CLF-CBF pair.

  \textbf{To show \ref{it:converse-safe-stab-clf-cbf} and
    \ref{it:converse-safe-stab-clf-ecbf}, we reason as follows.}
  Since $u_{\text{ss}}$ is a safe stabilizing controller,
  $\Cc$ is forward invariant and the origin is asymptotically stable
  for the closed-loop system $\dot{x}=f(x,u_{\text{ss}}(x))$ with
  $\Cc$ contained in an open set contained in its region of
  attraction. By~\cite[Theorem 4.17]{HK:02}, there exists a Lyapunov
  function $V$ for the closed-loop system
  $\dot{x}=f(x,u_{\text{ss}}(x))$, and
  $u_{\text{ss}}$ satisfies~\eqref{eq:cbf-ineq}.
  Now, if condition~\ref{it:zero} in Theorem~\ref{thm:converse-cbf}
  holds with $u_0 = u_{\text{ss}}$, there exists a CBF $h^*$ of
  $\Cc$. As shown in the proof of
  Theorem~\ref{thm:converse-cbf}\ref{it:zero},
  $u_{\text{ss}}$ satisfies the associated CBF
  condition~\eqref{eq:cbf-ineq} for all
  $x\in\Cc$.  Similarly, if condition~\ref{it:second} in
  Theorem~\ref{thm:converse-cbf} holds with $\hat{u} =
  u_{\text{ss}}$, there exists a CBF $h^*$ of
  $\Cc$.  As shown in the proof of
  Theorem~\ref{thm:converse-cbf}\ref{it:second},
  $u_{\text{ss}}$ satisfies the associated CBF
  condition~\eqref{eq:cbf-ineq} for all
  $x\in\Cc$.  Finally, if condition~\ref{it:third} in
  Theorem~\ref{thm:converse-cbf} holds, there exists a CBF $h^*$ of
  $\Cc$, and as shown in the proof of~\cite[Proposition
  3]{ADA-XX-JWG-PT:17}, any safe controller (in particular,
  $u_{\text{ss}}$) satisfies~\eqref{eq:cbf-ineq} for an appropriately
  defined extended class $\Kc_{\infty}$ function
  $\alpha$.  Hence, for every $x\in\Cc$,
  $u_{\text{ss}}(x)$ satisfies inequalities~\eqref{eq:clf-ineq}
  and~\eqref{eq:cbf-ineq}, which means that $V$ and
  $h^*$ are compatible, showing \ref{it:converse-safe-stab-clf-cbf}.

  Moreover, since $\Cc$ is safe under $u_{\text{ss}}$,
  Theorem~\ref{thm:converse-ecbf} implies that there exists an eCBF
  $\hat{h}$ of $\Cc$. Moreover, as shown in the proof of
  Theorem~\ref{thm:converse-ecbf}, any locally Lipschitz safe
  controller (in particular, $u_{\text{ss}}$)
  satisfies~\eqref{eq:ecbf-condition} for all $x\in\Cc$.  Since
  $u_{\text{ss}}(x)$ satisfies~\eqref{eq:clf-ineq}
  and~\eqref{eq:ecbf-condition} simultaneously, $V$ and
  $\hat{h}$ are a compatible CLF-eCBF pair, showing
  \ref{it:converse-safe-stab-clf-ecbf}.
\end{proof}

Theorem~\ref{thm:converse-safe-stab}~\ref{it:converse-safe-stab-clbf}
is consistent with Proposition~\ref{prop:topological-obstruction},
because it only ensures the existence of a CLBF if $\Cc$ is compact.

%
%

\begin{remark}\longthmtitle{On CLBFs and compatible pairs}
  {\rm 
  It is worth noting how Theorem~\ref{thm:converse-safe-stab}(i)-(iii)
  provide existence results under decreasingly restrictive
  hypotheses. In fact, the conditions in
  Theorem~\ref{thm:converse-safe-stab} under which a CLBF is
  guaranteed to exist always guarantee the existence of a compatible
  CLF-CBF pair, but the converse does not hold. To see this, note that
  from Proposition~\ref{prop:no-clbf-unbounded} and
  Theorem~\ref{thm:converse-safe-stab}, that unbounded sets containing
  a safely stabilizable point do not admit a CLBF, but they can admit
  a compatible CLF-CBF pair if either of the conditions in
  Theorem~\ref{thm:converse-cbf}\ref{it:zero},~\ref{it:second},
  or~\ref{it:third} hold.  Instead, compact safe sets that contain a
  safely stabilizable point and satisfy the strict inequality
  condition in
  Theorem~\ref{thm:converse-safe-stab}~\ref{it:converse-safe-stab-clbf}
  for some controller admit both a CLBF and a compatible CLF-CBF pair
  (because if the safe set is compact the assumptions of
  Theorem~\ref{thm:converse-safe-stab}~\ref{it:converse-safe-stab-clf-cbf}
  hold).  \demo
  }
\end{remark}




Next we address problem (P3) in
Section~\ref{sec:problem-statement}. 
%
The following example shows that in general, even if there exists a
CBF of $\Cc$ and a CLF on an open set containing $\Cc$, there might
not exist a strictly compatible CLF-CBF pair in $\Cc$.

\begin{example}\longthmtitle{Safety and stability
    separately do not imply safe
    stabilization}\label{ex:no-safe-stab}
  {\rm
  Consider the control-affine system:
  \begin{subequations}\label{eq:counterex-safe-stab}
    \begin{align}
      \dot{x}&=-x u,
      \\
      \dot{y}&=-y u.
    \end{align}
  \end{subequations}
  Let $h:\real^{2}\to\real$ be a differentiable function and $\Vc$ a
  neighborhood of $p=(0,5)$ with the following properties:
  \begin{itemize}
  \item $h(x,y)=1-(x+1)^{2}-(y-5)^{2}$ in $\Vc$,
  \item $h(0,0)>0$,
  \item $\setdef{x\in\real^n}{h(x)\geq0}$ is compact.
  \end{itemize}
  Let $\Cc$ be defined as in~\eqref{eq:safe-set}. The controller
  $u_{\text{sf}}:\real^{2}\to\real$ given by $u_{\text{sf}}(x,y)=0$
  for all $(x,y)\in\real^{2}$ renders the set $\Cc$ safe.  Since $\Cc$
  is compact, by Theorem~\ref{thm:converse-cbf}~\ref{it:third} it
  follows that $h$ is a CBF of $\Cc$.  Moreover, the origin is
  globally asymptotically stabilizable, since the controller
  $u_{\text{st}}:\real^2\to\real$ given by $u_{\text{st}}(x,y)=1$ for
  all $(x,y)\in\real^{2}$ makes the origin globally asymptotically
  stable.  Moreover, $V(x,y) = \frac{1}{2}(x^2 + y^2)$ is a CLF in
  $\real^2$.  However, any locally Lipschitz controller
  $\hat{u}:\real^{2}\to\real$ such that $\hat{u}(0,5)\neq0$ steers the
  trajectory starting at the point $(0,5)$ away from $\Cc$. Indeed,
  note that the $y$ axis is forward invariant and hence $x(t)=0$ for
  all $t\geq0$.  Moreover, the solution
  of~\eqref{eq:counterex-safe-stab} is differentiable and by
  performing a Taylor expansion of first order, for time
  $0<\epsilon\ll1$, the solution of~\eqref{eq:counterex-safe-stab}
  satisfies
  \begin{align*}
    y(\epsilon)=5-5u(p)\epsilon+O(\epsilon^2). 
  \end{align*}
  This implies that
  \begin{align*}
    h(x(\epsilon),y(\epsilon))=-25u(p)^2\epsilon^2+O(\epsilon^3),
  \end{align*}
  and therefore $h(x(\epsilon),y(\epsilon))<0$ for small enough
  $\epsilon$. Hence, there does not exist a safe stabilizing
  controller in $\Cc$. Therefore, even though $h$ is a CBF of $\Cc$
  and $V$ is a CLF in $\real^2$, there does not exist a strictly
  compatible CLF-CBF pair. Indeed, if that were the case the control
  design provided in~\cite{PO-JC:19-cdc} would yield a safe
  stabilizing controller, which does not exist.  Note that this
  example does not preclude the existence of a compatible CLF-CBF
  pair. However, even if such a compatible pair exists, one would not
  be able to use it to obtain a safe stabilizing controller.
  \problemfinal
  }
\end{example}


Note that the cause of difficulty in Example~\ref{ex:no-safe-stab} is
the point $p=(0,5)$, which is such that $\nabla h(x)^T f(x,u)=0$ for
any $u\in\real^{m}$.  Instead, using
Theorem~\ref{thm:converse-safe-stab}~\ref{it:converse-safe-stab-clbf},
we know that if there exists a locally Lipschitz controller
$u_{\text{str}}:\real^n\to\real^m$ such that
$\nabla h(x)^T f(x,u_{\text{str}}(x)) > 0$ for all $x\in\partial\Cc$,
$\Cc$ is compact and there exists a stabilizing controller with
region of attraction containing the safe set, then
there exists a strictly compatible CLF-CBF pair. Note
  that the proof of
  Theorem~\ref{thm:converse-safe-stab}~\ref{it:converse-safe-stab-clbf}
  heavily relies on the compactness of~$\Cc$.
  Next, we provide a similar result for non-compact $\Cc$ but
  restricted to control-affine systems.

\begin{proposition}\longthmtitle{Existence of compatible CLF-eCBF
    pair}\label{prop:compatible-clf-ecbf}
  Given an open set $\Gamma$ such that $\Cc\subseteq\Gamma$, let $h$
  be an eCBF of $\Cc$ with $0$ as a regular value and $V$ be a CLF on
  $\Gamma$. Further assume that the dynamics are control-affine, so
  that $\dot{x}=a(x)+g(x)u$, with $a:\real^{n}\to\real^{n}$ and
  $g:\real^{n}\to\real^{m}$ locally Lipschitz. Let
  $\Pc:=\setdef{x\in\real^n}{L_gV(x)=\kappa L_gh(x), \
    \kappa>0}$. Assume that
  $\Pc \cap \partial\Cc$ is contained in
  \begin{align*}
    \setdef{x\in\partial\Cc}{L_gh(x)\neq\mathbf{0}_m, \
    L_aV(x)<\frac{L_gV(x)^T L_gh(x)}{\norm{L_gh(x)}^2}L_ah(x)} .    
  \end{align*}
  Then, there exists a compatible CLF-eCBF pair in $\Cc$ with $0$ a
  regular value of the eCBF.
\end{proposition}
%
%
\begin{proof} The proof relies on the characterization
    of compatibility for a CLF and a CBF, as provided in~\cite[Lemma
    5.2]{PM-JC:23-csl}, and which naturally extends to CLFs and eCBFs.
    This result says that $V$ and $h$ are compatible at $x\in\Cc$ if
    and only if $L_gV(x)$ and $L_gh(x)$ are linearly dependent,
    $L_gV(x)^T L_gh(x) > 0$ and
    $L_fV(x) + W(x) > \frac{ L_gV(x)^T L_gh(x) }{ \norm{L_gh(x)}^2 }(
    L_fh(x)+\alpha(h(x)) )$.
  By continuity of $L_{g}V$ and $L_{g}h$, there exists a neighborhood
  $\Tc$ of $\partial\Cc$ such that
  $\Pc\cap\Tc\subseteq\setdef{x\in\real^n}{L_gh(x)\neq\mathbf{0}_m, \
    L_aV(x)<\frac{L_gV(x)^T
    L_gh(x)}{\norm{L_gh(x)}^2}L_ah(x)}$. By~\cite[Lemma
  5.2]{PM-JC:23-csl}, $V$ and $h$ are compatible in $\Cc\cap\Tc$.
  Next, define
  $\Sc=\setdef{x\in\Cc\backslash(\Tc\cup\{ \mathbf{0}_n \})}{L_gV(x)^T
    L_gh(x)=\mathbf{0}_m}$. Given $y\in\Sc$, if $L_gV(y)$ and $L_{g}h(y)$ are
  linearly independent, $V$ and $h$ are compatible at $y$
  (cf.~\cite[Lemma 5.2]{PM-JC:23-csl}). If instead $L_{g}V(y)$ and
  $L_{g}h(y)$ are linearly dependent, $L_{g}V(y)=L_{g}h(y)=\mathbf{0}_m$ and $V$
  and $h$ are compatible at $y$ because $V$ is a CLF and $h$ an eCBF.
  Moreover, there exists a neighborhood $\bar{\Sc}$ of $\Sc$ such that
  $V$ and $h$ are compatible in $\bar{\Sc}$. This is because for any
  $y\in\Sc$, if $L_gV(y)$ and $L_gh(y)$ are linearly independent,
  there exists a neighborhood $\Vc(y)$ of $y$ where $L_{g}V(z)$ and
  $L_{g}h(z)$ are linearly independent for all $z\in\Vc(y)$, and hence
  by~\cite[Lemma 5.2]{PM-JC:23-csl}, $V$ and $h$ are compatible at
  $z$. If instead $L_gV(y)$ and $L_gh(y)$ are linearly dependent, we
  can assume without loss of generality that $L_ah(y)+\alpha(h(y))>0$
  and $L_aV(y)+W(y)<0$ (otherwise, define
  $\tilde{\alpha}(s):=\frac{1}{2}\alpha(s)$ and
  $\tilde{W(x)}=\frac{1}{2}W(x)$), hence making $V$ and $h$ compatible
  in a neighborhood of $y$.  Now, we only need to show that $V$ and
  $h$ are compatible at $\Cc\backslash(\Tc\cup\bar{\Sc})$.  Define
  $S_{r,c}:=\setdef{x\in\real^n}{0\leq h(x)\leq r, \ \norm{x}\leq
    c+c_{\min}}$, where $c_{\min}$ is taken so that
  $S_{0,0}\neq\emptyset$, and define $\alpha$ as follows:
  \begin{align*}
    \alpha(r,c):= \sup_{x\in S_{r,c}\backslash (\Tc\cup\bar{\Sc})}
    \Bigl\{ (L_aV(x)+&W(x))\frac{\norm{L_gh(x)}^2}{L_gV(x)^T L_gh(x)}
    \\
                     &- L_ah(x) \Bigr\}.
  \end{align*}
  Since $S_{r,c}\backslash (\Tc\cup\bar{\Sc})$ is bounded for all
  $r\geq0$, $c\geq0$, and there exists a positive constant
  $\iota_{r,c}>0$ such that $L_gV(x)^T L_gh(x) > \iota_{r,c}$ for all
  $x\in S_{r,c}$, $\alpha(r,c)$ is finite for all $r\geq0$, $c\geq0$.
  Hence, there exists a class $\Kc_{\infty}\Kc$ function
  $\hat{\alpha}$ such that
  $\hat{\alpha}(h(x),\norm{x})\geq \alpha(h(x),\norm{x})$ for all
  $x\in \Cc\backslash(\Tc\cup\bar{\Sc})$.
  Hence, by~\cite[Lemma 5.2]{PM-JC:23-csl}, for all
  $x\in\Cc\backslash{(\Tc\cup\bar{\Sc})}$ there exists $u\in\real^{m}$
  satisfying
  \begin{align*}
    L_aV(x)+L_gV(x)u+W(x)\leq0,
    \\
    L_ah(x)+L_gh(x)u+\hat{\alpha}(h(x),\norm{x})\geq0,
  \end{align*}
  and hence $V$ and $h$ are compatible in all of $\Cc$.
\end{proof}

The conditions in Proposition~\ref{prop:compatible-clf-ecbf} are only
sufficient. In other words, there could exist weaker conditions
ensuring the existence of a compatible CLF-eCBF pair.

\begin{remark}\longthmtitle{Origin at the boundary of safe
    set}\label{rem:origin-boundary-safe-set} 
  \rm{The treatment above relies on the assumption that the origin
    belongs to $\text{Int}(\Cc)$. The extension of our results to the
    case when the origin is instead at $\partial\Cc$ remains an open
    problem.
    In fact, 
    establishing whether in such case there exists a CLBF of
    $\real^n\backslash\Cc$, or a (strictly) compatible CLF-CBF pair in
    $\Cc$ requires facing additional technical challenges.  For
    example, the construction of the CLBF in the proof of
    Theorem~\ref{thm:converse-safe-stab} relies on the fact that
    $\mathbf{0}_n$ does not belong to the set $\Tc$ and belongs to the
    set $\Pi$. Otherwise, the CLBF $\bar{V}$ (as defined therein) does
    not satisfy condition~\eqref{eq:clbf-negative-derivative} at the
    origin. } \demo
\end{remark}


\section{Conclusions}\label{sec:conclusions}

We have provided converse theorems on the existence of CBFs for the
study of safety and safe stabilization of control systems.  Regarding
safety, we have shown that for unbounded safe sets not all candidate
CBFs are CBFs, in contrast to what happens for bounded safe sets.
Next, we have derived a general set of conditions under which a CBF is
guaranteed to exist for any given safe set.  We have also extended the
definition of CBF conveniently to introduce eCBFs, and we have shown
that any safe set admits an eCBF.  Regarding safe stabilization, we
have established an alternate set of conditions under which a CLBF, a
(strictly) compatible CLF-CBF pair, and compatible CLF-eCBF pairs can
or can not exist.
Finally, we have shown via a counterexample that the existence of a
CLF and a CBF does not imply in general the existence of a strictly
compatible CLF-CBF pair, but we have found sufficient conditions under
which this holds.  Future work will focus on tightening the conditions
identified in our results and in extending the results to nonsmooth
barrier functions and discontinuous controlled dynamics.


\section*{Acknowledgments}
This work was partially supported by AFOSR Award FA9550-23-1-0740 and
ARO Award W911NF-23-1-0138. The authors wish to thank the anonymous
reviewers for their useful feedback and suggesting the condition in
Theorem~\ref{thm:converse-cbf}(i).

\bibliography{../bib/alias,../bib/JC,../bib/Main-add,../bib/Main}

\begin{thebibliography}{10}
\providecommand{\url}[1]{#1}
\csname url@samestyle\endcsname
\providecommand{\newblock}{\relax}
\providecommand{\bibinfo}[2]{#2}
\providecommand{\BIBentrySTDinterwordspacing}{\spaceskip=0pt\relax}
\providecommand{\BIBentryALTinterwordstretchfactor}{4}
\providecommand{\BIBentryALTinterwordspacing}{\spaceskip=\fontdimen2\font plus
\BIBentryALTinterwordstretchfactor\fontdimen3\font minus
  \fontdimen4\font\relax}
\providecommand{\BIBforeignlanguage}[2]{{%
\expandafter\ifx\csname l@#1\endcsname\relax
\typeout{** WARNING: IEEEtran.bst: No hyphenation pattern has been}%
\typeout{** loaded for the language `#1'. Using the pattern for}%
\typeout{** the default language instead.}%
\else
\language=\csname l@#1\endcsname
\fi
#2}}
\providecommand{\BIBdecl}{\relax}
\BIBdecl

\bibitem{PW-FA:07}
P.~Wieland and F.~Allg{\"o}wer, ``Constructive safety using control barrier
  functions,'' \emph{IFAC Proceedings Volumes}, vol.~40, no.~12, pp. 462--467,
  2007.

\bibitem{ADA-SC-ME-GN-KS-PT:19}
A.~D. Ames, S.~Coogan, M.~Egerstedt, G.~Notomista, K.~Sreenath, and P.~Tabuada,
  ``Control barrier functions: theory and applications,'' in \emph{{E}uropean
  {C}ontrol {C}onference}, Naples, Italy, 2019, pp. 3420--3431.

\bibitem{SCH-XX-ADA:15}
S.~Hsu, X.~Xu, and A.~D. Ames, ``Control barrier function based quadratic
  programs with applications to bipedal robot walking,'' in \emph{{A}merican
  {C}ontrol {C}onference}, Chicago, USA, July 2015.

\bibitem{ADA-XX-JWG-PT:17}
A.~D. Ames, X.~Xu, J.~W. Grizzle, and P.~Tabuada, ``Control barrier function
  based quadratic programs for safety critical systems,'' \emph{IEEE
  Transactions on Automatic Control}, vol.~62, no.~8, pp. 3861--3876, 2017.

\bibitem{TV-SM-RH-QH:21}
T.~Vu, S.~Mukherjee, R.~Huang, and Q.~Huang, ``Barrier function-based safe
  reinforcement learning for emergency control of power systems,'' in
  \emph{{IEEE} Conf.\ on Decision and Control}.\hskip 1em plus 0.5em minus
  0.4em\relax IEEE, 2021, pp. 3652--3657.

\bibitem{PM-JC:22-acc}
P.~Mestres and J.~Cort\'es, ``Safe design for controlling epidemic spreading
  under heterogeneous testing capabilities,'' in \emph{{A}merican {C}ontrol
  {C}onference}, Atlanta, Georgia, Jun. 2022, pp. 697--702.

\bibitem{SR:18}
S.~Ratschan, ``Converse theorems for safety and barrier certificates,''
  \emph{IEEE Transactions on Automatic Control}, vol.~63, no.~8, pp.
  2628--2632, 2018.

\bibitem{SP-AR:05}
S.~Prajna and A.~Rantzer, ``On the necessity of barrier certificates,''
  \emph{IFAC-PapersOnLine}, vol.~38, no.~1, pp. 526--513, 2005.

\bibitem{RW-CS:13}
R.~Wisniewski and C.~Sloth, ``Converse barrier certificate theorem,'' in
  \emph{{IEEE} Conf.\ on Decision and Control}, Firenze, Italy, 2013, pp.
  4713--4718.

\bibitem{JL:22}
J.~Liu, ``Converse barrier functions via {L}yapunov functions,'' \emph{IEEE
  Transactions on Automatic Control}, vol.~67, no.~1, pp. 497--503, 2022.

\bibitem{MM-RGS:23}
M.~Maghenem and R.~G. Sanfelice, ``On the converse safety problem for
  differential inclusions: Solutions, regularity, and time-varying barrier
  functions,'' \emph{IEEE Transactions on Automatic Control}, vol.~68, no.~1,
  pp. 172--187, 2023.

\bibitem{MM-MG:22}
M.~Maghenem and M.~Ghanbarpour, ``A converse robust-safety theorem for
  differential inclusions,'' \emph{arXiv preprint arXiv:2208.11364}, 2022.

\bibitem{MN:42}
M.~Nagumo, ``Über die {L}age der {I}ntegralkurven gewöhnlicher
  {D}ifferentialgleichungen,'' \emph{Proceedings of the Physico-Mathematical
  Society of Japan}, vol.~24, pp. 551--559, 1942.

\bibitem{PO-JC:19-cdc}
P.~Ong and J.~Cort\'es, ``Universal formula for smooth safe stabilization,'' in
  \emph{{IEEE} Conf.\ on Decision and Control}, Nice, France, Dec. 2019, pp.
  2373--2378.

\bibitem{KG-DP:21}
K.~Garg and D.~Panagou, ``Robust control barrier and control {L}yapunov
  functions with fixed-time convergence guarantees,'' in \emph{{A}merican
  {C}ontrol {C}onference}, New Orleans, LA, Jul. 2021, pp. 2292--2297.

\bibitem{PM-JC:23-csl}
P.~Mestres and J.~Cort\'es, ``Optimization-based safe stabilizing feedback with
  guaranteed region of attraction,'' \emph{IEEE Control Systems Letters},
  vol.~7, pp. 367--372, 2023.

\bibitem{MZR-BJ:16}
M.~Z. Romdlony and B.~Jayawardhana, ``Stabilization with guaranteed safety
  using control {L}yapunov-barrier function,'' \emph{Automatica}, vol.~66, pp.
  39--47, 2016.

\bibitem{PB-CMK:20}
P.~Braun and C.~M. Kellett, ``Comment on “{S}tabilization with guaranteed
  safety using control {L}yapunov–barrier function”,'' \emph{Automatica},
  vol. 122, p. 109225, 2020.

\bibitem{MK-DK:22}
M.~D. Kvalheim and D.~E. Koditschek, ``Necessary conditions for feedback
  stabilization and safety,'' \emph{Journal of Geometric Mechanics}, vol.~14,
  no.~4, pp. 659--693, 2022.

\bibitem{MDK:23}
M.~D. Kvalheim, ``Relationships between necesary conditions for feedback
  stabilizability,'' \emph{arXiv preprint arXiv:2312.16752}, 2023.

\bibitem{RWB:83a}
R.~W. Brockett, ``Asymptotic stability and feedback stabilization,'' in
  \emph{Geometric Control Theory}, R.~W. Brockett, R.~S. Millman, and H.~J.
  Sussmann, Eds.\hskip 1em plus 0.5em minus 0.4em\relax Birkh{\"a}user, 1983,
  pp. 181--191.

\bibitem{YM-YL-MF-JL:22}
Y.~Meng, Y.~Li, M.~Fitzsimmons, and J.~Liu, ``Smooth converse
  {L}yapunov-barrier theorems for asymptotic stability with safety constraints
  and reach-avoid-stay specifications,'' \emph{Automatica}, vol. 144, p.
  110478, 2022.

\bibitem{EDS:98}
E.~D. Sontag, \emph{Mathematical Control Theory: Deterministic Finite
  Dimensional Systems}, 2nd~ed., ser. TAM.\hskip 1em plus 0.5em minus
  0.4em\relax Springer, 1998, vol.~6.

\bibitem{RAF-PVK:96a}
R.~A. Freeman and P.~V. Kototovic, \emph{Robust Nonlinear Control Design:
  State-space and Lyapunov Techniques}.\hskip 1em plus 0.5em minus 0.4em\relax
  Cambridge, MA, USA: Birkhauser Boston Inc., 1996.

\bibitem{MA-NA-JC:25-tac}
M.~Alyaseen, N.~Atanasov, and J.~Cort\'es, ``Continuity and boundedness of
  minimum-norm {CBF}-safe controllers,'' \emph{IEEE Transactions on Automatic
  Control}, vol.~70, no.~6, 2025, to appear.

\bibitem{RK-ADA-SC:21}
R.~Konda, A.~D. Ames, and S.~Coogan, ``Characterizing safety: minimal control
  barrier functions from scalar comparison systems,'' \emph{IEEE Control
  Systems Letters}, vol.~5, no.~2, pp. 523--528, 2021.

\bibitem{EDS-HJS:80}
E.~D. Sontag and H.~J. Sussmann, ``Remarks on continuous feedback control,'' in
  \emph{{IEEE} Conf.\ on Decision and Control}, New Orleans, LA, USA, 1995, pp.
  2799--2805.

\bibitem{PM-AA-JC:25-ejc}
P.~Mestres, A.~Allibhoy, and J.~Cort\'es, ``Regularity properties of
  optimization-based controllers,'' \emph{European Journal of Control},
  vol.~81, p. 101098, 2025.

\bibitem{BL-WY-MC:22}
B.~Lin, W.~Yao, and M.~Cao, ``On {W}ilson's theorem about domains of attraction
  and tubular neigborhoods,'' \emph{Systems \& Control Letters}, vol. 167, p.
  105322, 2022.

\bibitem{JJC-DL-KS-CJT-SLH:21}
J.~J. Choi, D.~Lee, K.~Sreenath, C.~J. Tomlin, and S.~L. Herbert, ``Robust
  control barrier-value functions for safety-critical control,'' in
  \emph{{IEEE} Conf.\ on Decision and Control}, Austin, TX, USA, 2021, pp.
  6814--6821.

\bibitem{WR:87}
W.~Rudin, \emph{Real and Complex Analysis}, 3rd~ed.\hskip 1em plus 0.5em minus
  0.4em\relax McGraw-Hill, 1987.

\bibitem{YL-EDS-YW:96}
Y.~Lin, E.~D. Sontag, and Y.~Wang, ``A smooth converse {L}yapunov theorem for
  robust stability,'' \emph{SIAM Journal on Control and Optimization}, vol.~34,
  no.~1, pp. 124--160, 1996.

\bibitem{PB-CMK:17}
P.~Braun and C.~M. Kellett, ``On (the existence of) control {L}yapunov barrier
  functions,'' Newcastle, Australia, 2017, available at
  \url{https://epub.uni-bayreuth.de/id/eprint/3522/1/CLBFs_submission_pbraun.pdf}.

\bibitem{MFR-APA-PT:21}
M.~F. Reis, A.~P. Aguilar, and P.~Tabuada, ``Control barrier function-based
  quadratic programs introduce undesirable asymptotically stable equilibria,''
  \emph{IEEE Control Systems Letters}, vol.~5, no.~2, pp. 731--736, 2021.

\bibitem{XT-DVD:24}
X.~Tan and D.~V. Dimarogonas, ``On the undesired equilibria induced by control
  barrier function based quadratic programs,'' \emph{Automatica}, vol. 159, p.
  111359, 2024.

\bibitem{YC-PM-EDA-JC:24-cdc}
Y.~Chen, P.~Mestres, E.~Dall'Anese, and J.~Cort{\'e}s, ``Characterization of
  the dynamical properties of safety filters for linear planar systems,'' in
  \emph{{IEEE} Conf.\ on Decision and Control}, Milan, Italy, 2024, pp.
  2397--2402.

\bibitem{PM-YC-EDA-JC:25-jnls}
P.~Mestres, Y.~Chen, E.~Dall'Anese, and J.~Cort\'es, ``Control barrier
  function-based safety filters: characterization of undesired equilibria,
  unbounded trajectories, and limit cycles,'' \emph{Journal of Nonlinear
  Science}, 2025, submitted.

\bibitem{DEK:87-icra}
D.~E. Koditschek, ``Exact robot navigation by means of potential functions:
  some topological considerations,'' in \emph{{IEEE} Int. Conf.\ on Robotics
  and Automation}, Raleigh, {NC}, {USA}, 1987, pp. 1--6.

\bibitem{HK:02}
H.~Khalil, \emph{Nonlinear Systems, 3rd ed.}\hskip 1em plus 0.5em minus
  0.4em\relax Englewood Cliffs, NJ: Prentice Hall, 2002.

\end{thebibliography}
\bibliographystyle{IEEEtran}

\begin{IEEEbiography}[{\includegraphics[width=1in,height=1.25in,clip,keepaspectratio]{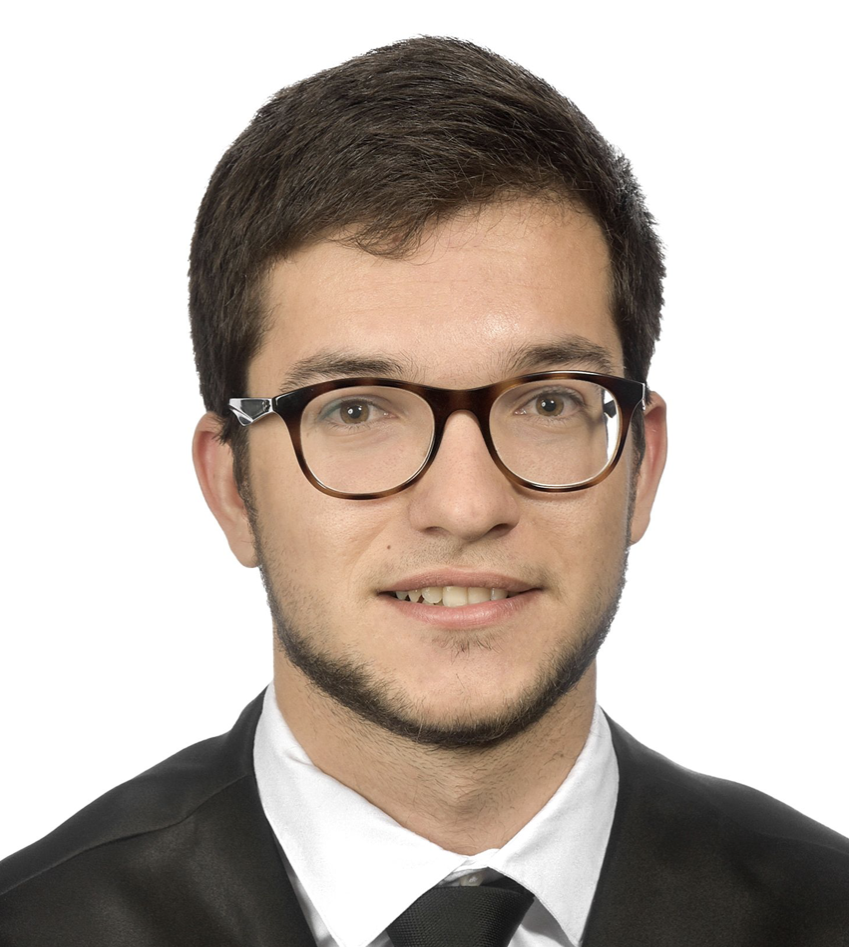}}]{Pol
    Mestres} received the Bachelor's degree in mathematics and the
  Bachelor's degree in engineering physics from the Universitat
  Polit\`{e}cnica de Catalunya, Barcelona, Spain, in 2020, and the
  Master's degree in mechanical engineering in 2021 from the
  University of California, San Diego, La Jolla, CA, USA, where he is
  currently a Ph.D candidate. His research interests include
  safety-critical control, optimization-based controllers, distributed
  optimization and motion planning.
\end{IEEEbiography}

\begin{IEEEbiography}[{\includegraphics[width=1in,height=1.25in,clip,keepaspectratio]{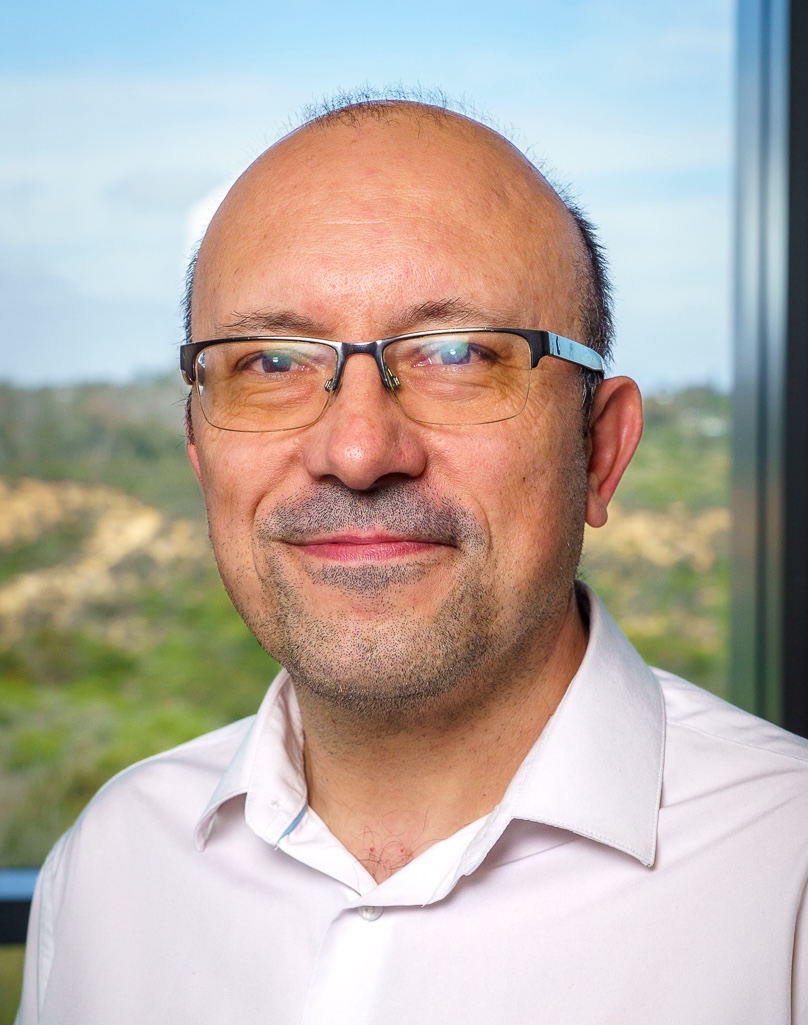}}]{Jorge
    Cort\'{e}s}(M'02, SM'06, F'14) received the Licenciatura degree in
  mathematics from Universidad de Zaragoza, Zaragoza, Spain, in 1997,
  and the Ph.D. degree in engineering mathematics from Universidad
  Carlos III de Madrid, Madrid, Spain, in 2001. He held postdoctoral
  positions with the University of Twente, Twente, The Netherlands,
  and the University of Illinois at Urbana-Champaign, Urbana, IL,
  USA. He was an Assistant Professor with the Department of Applied
  Mathematics and Statistics, University of California, Santa Cruz,
  CA, USA, from 2004 to 2007. He is a Professor and Cymer Corporation
  Endowed Chair in High Performance Dynamic Systems Modeling and
  Control at the Department of Mechanical and Aerospace Engineering,
  University of California, San Diego, CA, USA.  He is a Fellow of
  IEEE, SIAM, and IFAC.  His research interests include distributed
  control and optimization, network science, nonsmooth analysis,
  reasoning and decision making under uncertainty, network
  neuroscience, and multi-agent coordination in robotic, power, and
  transportation networks.
\end{IEEEbiography}

\end{document}